\documentclass[10pt]{article}

\usepackage{epsfig}
\usepackage{amsmath}
\usepackage{amsfonts}
\usepackage{amsthm}
\usepackage{amsbsy}
\usepackage{appendix}
\usepackage{epstopdf}
\usepackage[table]{xcolor}
\usepackage{subfloat}
\usepackage{float}
\usepackage[thinlines]{easytable}
\usepackage{array}
\allowdisplaybreaks
\numberwithin{equation}{section}

\theoremstyle{plain}

\newtheorem{theo}{Theorem}[section]

\theoremstyle{remark}
\newtheorem{remark}{\bf Remark}[section]
\theoremstyle{remark}

\theoremstyle{remark}

\newcommand{\Rho}{\mathrm{P}}

\title{A High Order Stochastic Asymptotic Preserving Scheme for Chemotaxis Kinetic Models with Random Inputs\footnote{This research was partially supported by
NSF grants DMS-1522184 and DMS-1107291: RNMS KI-Net, by NSFC
grant No. 91330203, and by the Office of the Vice Chancellor for Research and Graduate Education at the
University of Wisconsin-Madison with funding from the Wisconsin Alumni Research Foundation.}
\author{Shi Jin\footnote{Department of Mathematics, University of Wisconsin, Madison, WI 53706, USA (sjin{@}wisc.edu) and Institute of Natural Sciences, School of Mathematical Science, MOE-LSEC and SHL-MAC, Shanghai Jiao Tong
University, Shanghai 200240, China.} , Hanqing Lu\footnote{Department of Mathematics, University of Wisconsin, Madison, WI 53706, USA (hlu57{@}wisc.edu).} and Lorenzo Pareschi\footnote{Department of Mathematics \& Computer Science, University of Ferrara, Ferrara, 44121, Italy (lorenzo.pareschi{@}unife.it).}}} 

\begin{document}
\maketitle
\bigskip
{\bf Abstract}
In this paper, we develop a stochastic Asymptotic-Preserving (sAP) scheme for the kinetic chemotaxis system with random inputs, which will converge to the modified Keller-Segel model with random inputs in the diffusive regime. Based on the generalized Polynomial Chaos (gPC) approach, we design a high order stochastic Galerkin method using implicit-explicit (IMEX) Runge-Kutta (RK) time discretization with a macroscopic penalty term. The new schemes improve the parabolic CFL condition to a hyperbolic type when the mean free path is small, which shows significant efficiency especially in uncertainty quantification (UQ) with multi-scale problems. The stochastic Asymptotic-Preserving property will be shown asymptotically and verified numerically in several tests. Many other numerical tests are conducted to explore the effect of the randomness in the kinetic system, in the aim of providing more intuitions for the theoretic study of the chemotaxis models.
\\

\tableofcontents

\section{Introduction}
Chemotaxis is the movement of an organism in response to a chemical stimulus (called chemoattractant), approaching the regions of highest chemoattractant concentration. This process is critical to the early growth and subsequent development of the organism.

Mathematical study of this chemical system originates from the well-known (Patlak-)Keller-Segel model \cite{keller1970initiation,keller1971model,keller1971traveling,keller1980assessing,patlak1953random}. This model describes the drift-diffusion interactions between the cell density and chemoattractant concentration at a macroscopic level:
\begin{subequations}\label{ks1}
\begin{align}
&\partial_t\rho=\nabla\cdot (D\nabla \rho-\chi\rho\nabla s),\\
&\partial_t s=D_0\Delta s+q(s,\rho),
\end{align}
\end{subequations}
where $\rho(x,t)\geq 0$ is the cell density at position $x\in \mathbb R^n$ and time $t$, $s(x,t)\geq 0$ is the density of the chemoattractant, $D$ and $D_0$ are positive diffusive constants of the cells and the chemoattractant respectively, and $\chi$ is the positive chemotactic sensitivity constant. In (\ref{ks1}) the function $q(s,\rho)$ describes the interactions between the cell density and the chemoattractant such as productions and degradations. In the literature, several modifications and studies of the Keller-Segel model have been conducted during recent years, e.g. \cite{calvez2006modified,chertock2012chemotaxis,hillen2009user,horstmann20031970,perthame2004pde,perthame2006transport}. The one related to our study is the modified Keller-Segel model in \cite{calvez2006modified}:
\begin{subequations}\label{ks2}
\begin{align}
&\partial_t\rho=\nabla\cdot (D\nabla \rho-\chi\rho\nabla s),\\
&s=-\frac{1}{n\pi}\log|x|\ast \rho,
\end{align}
\end{subequations}
where $n$ is the space dimension. Notice that in $2$D, (\ref{ks1}) and (\ref{ks2}) are exactly the same if $q=0$.

An important property of the Keller-Segel system is the blow up behavior, which depends on the dimension of the system and the initial mass \cite{brenner1999diffusion,herrero1996chemotactic,nagai1997global,stevens1997aggregation}. For the $2$D Keller-Segel system (when (\ref{ks1}) and (\ref{ks2}) are equivalent), there exists a critical mass $M_c$ depending on the parameters of the system.  When the initial mass $M<M_c$ (subcritical case), global solution exists and presents a self-similar profile in long time; When the initial mass $M>M_c$ (supercritical case), the solution will blow up in finite time; When the initial mass $M=M_c$ (critical case), the solution will blow up in infinite time. This property can be extended to $1$D and $3$D for the modified Keller-Segel system (\ref{ks2}). The formula for the critical mass is given by
\begin{equation}\label{1-3}
M_c=\frac{2n^2\pi D}{\chi}.
\end{equation}

From another perspective, the chemotaxis can be described by a class of Boltzmann-type kinetic equations at a microscopic level. The kinetic description of the phase space cell density was first introduced by Alt \cite{alt1980orientation,alt1980biased} via a stochastic interpretation of the ``run" and ``tumble" process of bacteria movements. Later on Othmer, Dunbar and Alt formulated the following non-dimensionalized  chemotaxis kinetic system with parabolic scaling in  \cite{othmer1988models}:
\begin{equation}\label{1-4}
\varepsilon\frac{\partial f}{\partial t}+v\cdot \nabla_x f=\frac{1}{\varepsilon}\int_V(T_\varepsilon f'-T_\varepsilon^*f)dv'.
\end{equation}
Here $f(t,x,v)$ is the density function of cells at time $t\in \mathbb R^+$, position $x\in\mathbb R^n$ and moving with velocity $v\in V$, $V$ is a finite subset of $\mathbb R^n$. The small parameter $\varepsilon$ is the radio of the mean running length between jumps to the typical observation length scale and $f'$ is the abbreviation for $f(t,x,v')$. $T_\varepsilon=T_\varepsilon[s](t,x,v,v')$ with the property $T^*_\varepsilon[s](t,x,v,v')=T_\varepsilon[s](t,x,v',v)$, is the turning kernel operator depending on the density of chemoattractant $s(t,x)$, which also solves the Poisson equation (\ref{ks1}b).

The relationship between the kinetic chemotaxis model (\ref{1-4}) and the Keller-Segel model (\ref{ks1}) was formally derived by Othmer and Hillen in \cite{othmer2000diffusion,othmer2002diffusion} using moment expansions. Then Chalub et al. gave a rigorous proof that the Keller-Segel system (\ref{ks2}) (before blow up time in supercritical case and for all time in subcritical case) is the macroscopic limit (as $\varepsilon\to 0$) of the kinetic chemotaxis system (\ref{1-4}) coupled with (\ref{ks2}b) in three dimensions \cite{chalub2004kinetic}. For certain type of turning kernel $T_\varepsilon$ (the nonlocal model in Section $2.1$), \cite{chalub2004kinetic} also proved the global existence of the solution to the kinetic systems (\ref{1-4}) for any initial conditions, which behaves completely differently from the Keller-Segel system. For other types of of turning kernel $T_\varepsilon$ (e.g. the local model in Section $2.2$), many questions are unsolved yet. Blow up may happen with supercritical initial mass but the critical mass is different from the Keller-Segel equations \cite{bournaveas2009critical}. The long time behavior of the subcritical case is unclear yet. Also, theoretic proof of the blow up in the $1$D case is not available \cite{sharifi2011one}.

The microscopic kinetic model, with interesting properties and mysterious behaviors, make it appealing to investigate the system numerically. Moreover, the global existence of the solution with nonlocal turning kernel could help us to understand the behavior of chemotaxis after Keller-Segel solutions blow up. One of the difficulties in solving the kinetic chemotaxis model, as other multi-scale kinetic equations, is the stiffness when $0<\varepsilon\ll1$. Classical algorithms require taking spatial and time step of $O(\varepsilon)$, thus causing unaffordable computational cost. To overcome this difficulty, one has to design an \textit{Asymptotic-Preserving} (AP) scheme, which discretizes the kinetic equations with mesh and time step independent of $\varepsilon$ and  preserves a consistent discretization of the limiting modified Keller-Segel equation as $\varepsilon\to 0$. The AP methods were first coined in \cite{jin1999efficient} and have been applied to a variety of multi-scale kinetic equations. We refer to \cite{degond2011asymptotic,degond2017asymptotic,dimarco2014numerical,jin2010asymptotic} for detailed reviews on AP schemes. In particular, AP schemes have been designed to solve $1$D and $2$D kinetic chemotaxis model in \cite{carrillo2013asymptotic,chertockasymptotic}, which are most relevant to our study.

The main issue we want to address in this paper is the uncertainties involved in the kinetic model due to modeling and experimental errors. For example, different turning kernels are proposed as operators that mimic the ``run" and ``tumble" process of cell movements and thus may contain uncertainties. Moreover, initial and boundary data, or other coefficients in the equations could also be measured inaccurately. In such a system that behaves so sensitively to initial mass and turning kernel, only by quantifying the \textit{intrinsic} uncertainties in the model, could one get a better understanding and a more reliable prediction on the chemotaxis from computational simulations, especially in the situation where many properties are not clarified by theoretic study. 

The goal of this paper is to design a high order efficient numerical scheme such that uncertainty quantification (UQ) can be easily conducted. Only recently, studies in UQ begin to develop for kinetic equations \cite{hu2016stochastic,jin2016august,jin2017asymptotic,jinlu,jin2015asymptotic,zhu2016vlasov,crouseilles2017nonlinear}. To deal with numerical difficulties for uncertainty and multi-scale at the same time, the \textit{stochastic Asymptotic-Preserving} (sAP) notion was first introduced in \cite{jin2015asymptotic}. Since then, the \textit{generalized Polynomial Chaos} (gPC) based \textit{Stochastic Galerkin} (SG) framework has been developed to a variety of kinetic equations \cite{jinlu,jin2015asymptotic,zhu2016vlasov,crouseilles2017nonlinear}. In this paper, we are going to conduct UQ under the same gPC-SG framework, which projects the uncertain kinetic equations into a vectorized deterministic equations and thus allowing us to extend the deterministic AP solver in \cite{carrillo2013asymptotic}. The sAP property is going to be verified formally by showing that the kinetic chemotaxis model with uncertainty after SG projection in fully discrete setting, as $\varepsilon\to 0$, automatically becomes a numerical discretization of the Keller-Segel equations with uncertainty after the SG projection. As realized in \cite{jin2015asymptotic} and rigorously proved in \cite{jin2016august,li2016uniform,jin2017hypocoercivity}, the spectral accuracy is expected using this gPC-SG method as long as the regularity of the solution (which is usually preserved from initial regularity in kinetic equations) behaves well.

In addition, we improve the accuracy and efficiency of the numerical scheme by using the implicit-exlicit (IMEX) Runge-Kutta (RK) methods (see \cite{boscarino2013implicit,boscarino2017unified,pareschi2005implicit} and the references therein) and macroscopic penalization method. A similar approach was utilized in our previous work \cite{jin2017efficient} for linear transport and radiative heat transfer equations with random inputs. In \cite{jin2017efficient}, we improved the parabolic CFL condition $\Delta t=O((\Delta x)^2)$ in \cite{jin2015asymptotic} to a hyperbolic CFL condition $\Delta t=O(\Delta x)$, which allows to save the computational time significantly.

The rest of the paper is organized as follows. In section 2, the kinetic models with random inputs of two different  turning kernels are described and the macroscopic limits of both models are formally derived. From section 3 to section 5, the numerical scheme for the kinetic chemotaxis equations are designed and the sAP properties are illustrated. In section 6, several numerical tests are presented to illustrate the accuracy and efficiency of our scheme. The sAP property is also verified numerically.  Different properties, e.g. blow up, stationary solutions etc., influenced by the introduced randomness of the local and nonlocal model, are explored for the chemotaxis system. The interactions between peaks involved with different sources of uncertainty are compared to show the dynamics. Finally, some conclusions are drawn in section 7.

\section{The Kinetic Descriptions for Chemotaxis}
The chemotaxis kinetic system with random inputs we are going to study is (\ref{1-4}) coupled with (\ref{ks2}b) in 1D:
\begin{subequations}\label{2-1}
\begin{align}
&\varepsilon\frac{\partial f}{\partial t} +v\frac{\partial f}{\partial x}=\frac{1}{\varepsilon}\int_V(T_\varepsilon f'-T_\varepsilon^*f)dv',\\
&s=-\frac{1}{\pi}\log|x|\ast\rho, \ \ \ \rho=\int_Vfdv,
\end{align}
\end{subequations}
where
$x\in \Omega=[-x_{\text{max}},x_{\text{max}}]\subset \mathbb R$, $v\in V=[-v_{\text{max}},v_{\text{max}}]\subset \mathbb R$.

The only difference is now $f=f(t,x,v,z)$ and $s=s(t,x,z)$ have dependence on the random variable $z\in I_z\subset\mathbb R^d(d\geq 1)$ with compact support $I_z$, in order to account for random uncertainties.

Now we specify the turning kernel operator $T_\varepsilon$ in (\ref{2-1}). Since the turning kernel $T_\varepsilon[s](t,x,z,v,v')$ measures the probability of velocity jump of cells from $v$ to $v'$, it has the following properties
\begin{equation}\label{2-2}
\begin{aligned}
&T_\varepsilon[s](t,x,z,v,v')\geq 0,\\
&T_\varepsilon[s](t,x,z,v,v')=F(z,v)+\varepsilon T_1+O(\varepsilon^2),
\end{aligned}
\end{equation}
where $F(z,v)$ is the equilibrium of velocity distribution and $T_1\geq 0$ characterizes the directional preference.

\subsection{The 1D Nonlocal Model}
Now considering the nonlinear kernel introduced in \cite{chalub2004kinetic} with uncertainty,
\begin{equation}\label{2-3}
T_\varepsilon[s](t,x,z,v,v')=\alpha_+(z)\psi(s(t,x,z),s(t,x+\varepsilon v,z))+\alpha_-(z)\psi(s(t,x,z),s(t,x-\varepsilon v',z)).
\end{equation}
The first term describes the cell movement to a new direction decided by the detection of current environment and probable new location and the second term describes the influence of the past memory on the choice of the new moving direction.

For simplicity, the past memory influence is neglected. Since $\alpha_+$ is an experimental parameter, we introduce the randomness on $\alpha_+(z)>0$ with the probability density function $\lambda(z)$ for the random variable $z$ and take
\begin{equation}\label{2-4}
\psi(s(t,x,z),s(t,x+\varepsilon v,z))=\bar F(v)+\delta^\varepsilon s(x,z,v),
\end{equation} 
where 
\begin{equation}\label{2-5}
\delta^\varepsilon s(x,z,v)=(s(t,x+\varepsilon v,z)-s(t,x,z))_+:=\left\{
\begin{aligned}
&s(t,x+\varepsilon v,z)-s(t,x,z)&&\text{if}\ \ s(t,x+\varepsilon v,z)-s(t,x,z)>0\\
&0&&\text{otherwise}
\end{aligned}
\right.,
\end{equation}
and $\bar F(v)$ satisfies
\begin{equation}\label{2-6}
\left\{
\begin{aligned}
&\int_V\bar F(v)dv=1,\\
&\bar F(v)=\bar F(|v|).
\end{aligned}
\right.
\end{equation}
Notice that $\delta^\varepsilon s$ is an $O(\varepsilon)$ term which corresponds to $\varepsilon T_1$ in (\ref{2-2}).

Then the kinetic system (\ref{2-1}) becomes
\begin{subequations}\label{2-7}
\begin{align}
&\varepsilon \frac{\partial f}{\partial t}+v\frac{\partial f}{\partial x}=\frac{\alpha_+(z)}{\varepsilon}\left[(\bar F(v)+\delta^\varepsilon s(v))\rho-\left(1+\int_V\delta^\varepsilon s(v')dv'\right)f\right],\\
&s=-\frac{1}{\pi}\log|x|\ast \rho.
\end{align}
\end{subequations}

Positive initial conditions and reflection boundary conditions for $f$, reflecting boundary conditions for $s$ are imposed as following:
\begin{subequations}\label{2-8}
\begin{align}
&f(0,x,z,v)=f^I(x,z,v)\geq 0,\\
&s(0,x,z)=s^I(x,z)\geq 0,\\
&f(t,\pm x_{\text{max}},z,v)=f(t,\pm x_{\text{max}},z,-v),\\
&\partial_xs|_{x=\pm x_{\text{max}}}=0.
\end{align}
\end{subequations}
\begin{remark}
The global existence of the solution to (\ref{2-7}) for fixed $z$ with any initial mass is proved in \cite{chalub2004kinetic}.
\end{remark}
\subsection{The 1D Local Model}
For the local model, we consider the turning kernel introduced in \cite{bournaveas2009critical} with uncertainty,
\begin{equation}\label{2-9}
T_\varepsilon=T_\varepsilon[s](t,x,z,v,v')=\alpha_+(z)\left[\bar F(v)+\varepsilon(v\cdot \nabla s (x))_+\right],
\end{equation}
where $\bar F$ is the equilibrium function satisfying (\ref{2-6}) and $\alpha(z)>0$ describes the desire of the cell to change to a favorable direction, which could come with uncertainty. Similarly as in section 2.1, we introduce the randomness on $\alpha_+(z)>0$. Then the kinetic equation (\ref{2-1}) in one dimension is
\begin{subequations}\label{2-10}
\begin{align}
&\varepsilon\frac{\partial f}{\partial t}+v\frac{\partial f}{\partial x}=\frac{\alpha_+}{\varepsilon}\left[(\bar F(v)+\varepsilon(v\cdot \nabla s)_+)\rho-(1+c_1\varepsilon|\nabla s|)f\right],\\
&s=-\frac{1}{\pi}\log|x|\ast \rho,
\end{align}
\end{subequations}
with $c_1=\int_V(v\cdot \nabla s/|\nabla s|)_+dv=\frac{1}{2}\int_V|v|dv$. The same initial and boundary conditions in (\ref{2-8}) are applied.

\subsection{The Macroscopic Limits}
The nonlocal kinetic model (\ref{2-7}) and the local one (\ref{2-10}) give the same asymptotic limit when $\varepsilon\to 0$. Inserting the Hilbert expansion into (\ref{2-7}a) and (\ref{2-10}a) and collecting the same order terms, one can derive the classical modified Keller-Segel system for $\rho$ as $\varepsilon\to 0$:
\begin{subequations}\label{2-11}
\begin{align}
&\partial_t \rho=\partial_x\left(\frac{D}{\alpha_+}\partial_x\rho-\chi\rho\partial_x s\right),\\
&s=-\frac{1}{\pi}\log|x|\ast \rho,\\
&\partial_x\rho|_{x=\pm x_{\text{max}}}=0,\\
&\partial_xs|_{x=\pm x_{\text{max}}}=0,
\end{align}
\end{subequations}
where
\begin{equation}\label{2-12}
D=\int_V|v|^2\bar F(v)dv, \ \ \chi=\frac{1}{2}\int_V|v|^2dv.
\end{equation}
We refer to \cite{chalub2004kinetic} for the details.

\subsection{The Critical Mass with Random Inputs}
To derive the critical mass for system (\ref{2-11}), we show, following \cite{calvez2006modified}, that the second momentum (with respect to $x$) of $\rho$ cannot remain positive for all time. 

We use 

$$\partial_x s=\partial_x(-\frac{1}{\pi}\log|x|\ast\rho)=-\frac{1}{\pi}\int_\Omega\frac{1}{x-y}\rho(y)dy=-\mathcal H \rho,$$

 where $\mathcal H$ denotes the Hilbert transform \cite{zuazo2001fourier}. Then
\begin{equation}\label{2-13}
\begin{aligned}
\frac{d}{dt}\int_\Omega\frac{1}{2}|x|^2\rho(x,z,t)dx=&\int_\Omega\frac{1}{2}|x|^2\frac{\partial \rho}{\partial t}dx\\
=&\int_\Omega\frac{1}{2}|x|^2\partial_x\left(\frac{D}{\alpha_+(z)}\partial_x\rho-\chi\rho\partial_xs\right)dx\\
=&-\int_\Omega x\left(\frac{D}{\alpha_+(z)}\partial_x\rho-\chi\rho\partial_xs\right)dx\\
=&-\frac{D}{\alpha_+(z)}[x_{\text{max}}\rho(x_{\text{max}})+x_{\text{max}}\rho(-x_{\text{max}})]+\frac{D}{\alpha_+(z)}M\\
&-\frac{\chi}{\pi}\int_\Omega \rho(x)\lim_{\delta \to 0}\int_{|x-y|>\delta}\frac{x}{x-y}\rho(y)dydx\\
=&-\frac{D}{\alpha_+(z)}x_{\text{max}}[\rho(x_{\text{max}})+\rho(-x_{\text{max}})]+\frac{D}{\alpha_+(z)}M\\
&-\frac{\chi}{2\pi}\lim_{\delta \to 0}\int_\Omega\int_{|x-y|>\delta}\rho(x)\rho(y)dxdy\\
=& -\frac{D}{\alpha_+(z)}x_{\text{max}}[\rho(x_{\text{max}})+\rho(-x_{\text{max}})]-\frac{\chi}{2\pi}M^2\left(1-\frac{M_c(z)}{M}\right),
\end{aligned}
\end{equation}
where
\begin{equation}\label{2-14}
M_c(z)=\frac{2\pi D}{\chi \alpha_+(z)}.
\end{equation}
Here we assume that the initial data is independent of $z$ and we use the conservation of mass, i.e. $M=\int_\Omega \rho dx$ is a constant independent of $z$.

When $M>M_c(z)$, $\frac{d}{dt}\int_\Omega\frac{1}{2}|x|^2\rho(x,z,t)dx\leq -c<0$, where $c$ is a positive constant. To preserve the positivity of this second moment (with respect to $x$), some singularity has to occur so that the above computation will not hold at certain time. The singularity is rigorously analyzed in \cite{dolbeault2004optimal,blanchet2006two} and $\partial_xs$ is unbounded in this case. Thus blow up occurs.

When $M< M_c(z)$, the second moment (with respect to $x$) is locally controlled and global existence of weak solution can be obtained \cite{calvez2006modified}.

\begin{remark}
When $n\geq 2$, the computation is similar and the general formular for $M_c(z)$ is 
$$M_c(z)=\frac{2n^2\pi D}{\chi\alpha_+(z)}.$$
\end{remark}

In practice, one is more interested in the behavior of $\mathbb{E}[\rho(x,z,t)]$, the expected value of $\rho(x,z,t)$. We have the following theorem analyzing the influence of initial mass on $\mathbb{E}[\rho(x,z,t)]$.

\begin{theo}
Suppose that the total mass $M$ is independent of $z$. Denote $\bar M_c$ as the critical mass for $\mathbb{E}[\rho(x,z,t)]$, i.e. when $M>\bar M_c$, $\mathbb{E}[\rho(x,z,t)]$ will blow up; when $M<\bar M_c$, $\mathbb{E}[\rho(x,z,t)]$ will be bounded for all time. Then we have
\begin{equation}\label{2-15}
\bar{M}_c=\mathbb{E}[M_c(z)].
\end{equation}
\end{theo}
\begin{proof}
Following the computations in (\ref{2-13}), we show that
\begin{equation}\label{2-16}
\begin{aligned}
\frac{d}{dt}\int_\Omega\frac{1}{2}|x|^2\mathbb{E}[\rho(x,z,t)]dx=&\int_\Omega\int_{I_z}\frac{1}{2}|x|^2\frac{\partial \rho(x,z,t)}{\partial t}\lambda(z)dzdx\\
=&\int_\Omega \int_{I_z}\frac{1}{2}|x|^2\partial_x \left(\frac{D}{\alpha_+(z)}\partial_x \rho-\chi \rho\partial_x s\right)\lambda(z)dzdx\\
=&\int_{I_z}\left[\int_\Omega \frac{1}{2}|x|^2\partial_x \left(\frac{D}{\alpha_+(z)}\partial_x\rho-\chi \rho\partial_xs\right)dx\right]\lambda(z)dz\\
=&\int_{I_z}\left[-\frac{D}{\alpha_+(z)}[x_{\text{max}}\rho(x_{\text{max}})+x_{\text{max}}\rho(-x_{\text{max}})]-\frac{\chi}{2\pi}M^2\left(1-\frac{M_c(z)}{M}\right)\right]\lambda(z)dz\\
=&-\int_{I_z}\frac{D}{\alpha_+(z)}x_{\text{max}}[\rho(x_{\text{max}})+\rho(-x_{\text{max}})]\lambda(z)dz-\frac{\chi}{2\pi}M^2\left(1-\frac{\mathbb{E}[M_c(z)]}{M}\right)\\
\leq&-\frac{\chi}{2\pi}M^2\left(1-\frac{\mathbb{E}[M_c(z)]}{M}\right).
\end{aligned}
\end{equation}
Thus, $\bar{M}_c=\mathbb{E}[M_c(z)]$ is the critical mass for $\mathbb{E}[\rho(x,z,t)]$.
\end{proof}
\begin{remark}
The same conclusion holds for $n\geq 2$.
\end{remark}

\section{The Even-Odd Decomposition}
In this section, we apply the even-odd decomposition to reformulate the problem following the same procedure as \cite{carrillo2013asymptotic} for deterministic kinetic model for chemotaxis.

\subsection{The 1D Nonlocal Model}
For $v>0$, (\ref{2-7}a) can be split into two equations:
\begin{subequations}\label{3-2}
\begin{align}
&\varepsilon \frac{\partial f(v)}{\partial t}+v\frac{\partial f(v)}{\partial x}=\frac{\alpha_+(z)}{\varepsilon}\left[(\bar F(v)+\delta^\varepsilon s(v))\rho-\left(1+\int_V\delta^\varepsilon s(v')dv'\right)f(v)\right],\\
&\varepsilon \frac{\partial f(-v)}{\partial t}-v\frac{\partial f(-v)}{\partial x}=\frac{\alpha_+(z)}{\varepsilon}\left[(\bar F(-v)+\delta^\varepsilon s(-v))\rho-\left(1+\int_V\delta^\varepsilon s(v')dv'\right)f(-v)\right].
\end{align}
\end{subequations}
Now denote the even and odd parities
\begin{subequations}\label{3-3}
\begin{align}
&r(t,x,z,v)=\mathcal R[f]=\frac{1}{2}(f(t,x,z,v)+f(t,x,z,-v)),\\
&j(t,x,z,v)=\mathcal J[f]=\frac{1}{2\varepsilon}(f(t,x,z,v)-f(t,x,z,-v)).
\end{align}
\end{subequations}
Then (\ref{3-2}) becomes
\begin{subequations}\label{3-4}
\begin{align}
&\partial_t r+v\partial_xj=\frac{\alpha_+}{\varepsilon^2}[(\bar F(v)+\mathcal R[\delta^\varepsilon s])\rho-(1+\langle \delta^\varepsilon s\rangle)r],\\
&\partial_t j+\frac{1}{\varepsilon^2}v\partial_x r=\frac{\alpha_+}{\varepsilon^2}(\mathcal J[\delta^\varepsilon s]\rho-(1+\langle \delta^\varepsilon s\rangle)j),
\end{align}
\end{subequations}
where
\begin{subequations}\label{3-5}
\begin{align}
&\langle\delta^\varepsilon s\rangle=\int_V\delta^\varepsilon s(x,v')dv',\\
&\rho=\int_Vfdv=2\int_{V^+}rdv, \ V^+=\{v\in V|v\geq 0\}.
\end{align}
\end{subequations}
Notice that, when $\varepsilon\to 0$, (\ref{3-4}) yields
\begin{subequations}\label{3-6}
\begin{align}
&r=\frac{\bar F(v)+\mathcal R[\delta^\varepsilon s]}{1+\langle \delta^\varepsilon s\rangle}\rho=\rho \bar F(v)+O(\varepsilon),\\
&j=\frac{\mathcal J[\delta^\varepsilon s]\rho-v\frac{\partial_x r}{\alpha_+}}{1+\langle \delta^\varepsilon s\rangle}=v\left(\frac{1}{2}\partial_x s\rho-\frac{\partial_x r}{\alpha_+}\right)+O(\varepsilon).
\end{align}
\end{subequations}
Substituting (\ref{3-6}) into (\ref{3-4}a) and integrating over $V^+$, one gets the same limiting Keller-Segel equations with random inputs as (\ref{2-11}).

\subsection{The 1D Local Model}
For the 1D local model, one can follow the same even-odd decomposition and obtain
\begin{subequations}\label{3-7}
\begin{align}
&\partial_t r+v\partial_x j=\frac{\alpha_+}{\varepsilon}\left[(\bar F(v)+\frac{\varepsilon}{2}|v\partial_x s|)\rho-(1+c_1\varepsilon|\partial_xs|)r\right],\\
&\partial_t j+\frac{1}{\varepsilon^2}v\partial_x r=\frac{\alpha_+}{\varepsilon^2}\left[\frac{1}{2}v\partial_x s\rho-(1+c_1\varepsilon|\partial_xs|)j\right].
\end{align}
\end{subequations}
The remaining work is the same as section 3.1.

\section{The gPC-SG Formulation}
Now we deal with the random inputs using the gPC expansion via an orthogonal polynomial series to approximate the solution. That is, for random variable $z\in \mathbb R^d$, one seeks
\begin{subequations}\label{4-1}
\begin{align}
&r(t,x,z,v)\approx r_N(t,x,z,v)=\sum_{k=1}^K\hat{r}_k(t,x,v)\Phi_k(z),\\
&j(t,x,z,v)\approx j_N(t,x,z,v)=\sum_{k=1}^K\hat{j}_k(t,x,v)\Phi_k(z),
\end{align}
\end{subequations}
where $\left\{\Phi_k(z),1\leq k\leq K, K=\begin{pmatrix}
d+N\\
d
\end{pmatrix}\right\}$ are from $\mathbb P_N^d$, the $d$-variate orthogonal polynomials of degree up to $N\geq 1$, and orthonormal
\begin{equation}\label{4-2}
\int_{I_z}\Phi_i(z)\Phi_j(z)\lambda(z)dz=\delta_{ij}, \ 1\leq i,j\leq K=\text{dim}(\mathbb P_N^d).
\end{equation}
Here $\delta_{i,j}$ the Kronecker delta function (See \cite{xiu2002wiener}).

Now inserts the approximation (\ref{4-1}) into the governing equation (\ref{3-4}) and enforces the residue to be orthogonal to the polynomial space spanned by $\{\Phi_1,\cdots, \Phi_K\}$. Thus, we obtain a set of vector deterministic equations for $\hat{\bold r}=(\hat r_1,\cdots, \hat r_K)^T$, $\hat{\bold j}=(\hat j_1,\cdots, \hat j_K)^T$ and $\hat{\bold s}=(\hat s_1,\cdots, \hat s_K)^T$:
\begin{subequations}\label{4-3}
\begin{align}
&\partial_t \hat{\bold r}+v\partial_x\hat{\bold j}=\frac{1}{\varepsilon^2}[\bar F(v)\bold M \hat{\boldsymbol \rho}+\bold B\hat{\boldsymbol \rho}-\bold M\hat{\bold r}-\bold C\hat{\bold r}],\\
&\partial_t \hat{\bold j}+\frac{1}{\varepsilon^2}v\partial_x \hat{\bold r}=\frac{1}{\varepsilon^2}(\bold E\hat{\boldsymbol \rho}-\bold M\hat{\bold j}-\bold C\hat{\bold j}),\\
&\hat{\bold s}=-\frac{1}{\pi}\log|x|\ast \hat{\boldsymbol\rho},
\end{align}
\end{subequations}
where
\begin{equation}\label{4-4}
\hat{\boldsymbol \rho}(t,x)=\langle \hat{ \bold r}\rangle=2\int_{V^+}\hat{\bold r}dv,
\end{equation}
and $\bold M=(m_{ij})_{1\leq i,j\leq K}$, $\bold B(\delta^\varepsilon s_N)=(b_{ij}(x,v))_{1\leq i,j\leq K}$, $\bold C(\langle\delta^\varepsilon s_N\rangle )=(c_{ij}(x))_{1\leq i,j\leq K}$ and $\bold E(\delta^\varepsilon s_N)=(e_{ij}(x,v))_{1\leq i,j\leq K}$ are $K\times K$ symmetric matrices with entries respectively
\begin{subequations}\label{4-5}
\begin{align}
&m_{ij}=\int_{I_z}\alpha_+(z)\Phi_i(z)\Phi_j(z)\lambda(z)dz,\\
&b_{ij}(x,v)=\int_{I_z}\alpha_+(z)\mathcal R[\delta^\varepsilon s_N]\Phi_i(z)\Phi_j(z)\lambda(z)dz,\\
&c_{ij}(x)=\int_{I_z}\alpha_+(z)\langle\delta^\varepsilon s_N\rangle\Phi_i(z)\Phi_j(z)\lambda(z)dz,\\
&e_{ij}(x,v)=\int_{I_z}\alpha_+(z)\mathcal J[\delta^\varepsilon s_N]\Phi_i(z)\Phi_j(z)\lambda(z)dz.
\end{align}
\end{subequations}

As $\varepsilon\to 0^+$ in (\ref{4-3}), since $\langle \delta^\varepsilon s_N\rangle=O(\varepsilon)$ and the matrices $\bold M$ and $\bold C$ are symmetric positive definite thus invertible,
\begin{subequations}\label{4-6}
\begin{align}
&\hat{\bold r}=(\bold M+\bold C)^{-1}(\bar F(v)\bold M+\bold B)\hat{\boldsymbol \rho}=\bar{F}(v)\hat{\boldsymbol \rho}+O(\varepsilon),\\
&\hat{\bold j}=(\bold M+\bold C)^{-1}(\bold E\hat{\boldsymbol \rho}-v\partial_x\hat{\bold r})=\bold M^{-1}\bold E\hat{\boldsymbol \rho}-v\bold M^{-1}\partial_x\hat{\boldsymbol r}+O(\varepsilon).
\end{align}
\end{subequations}
Plugging (\ref{4-6}) into (\ref{4-3}a) and integrating over $V^+$, one obtains 
\begin{equation}\label{4-7}
\partial_x\hat{\boldsymbol \rho}=\partial_x\left(D\bold M^{-1}\partial_x \hat{\boldsymbol \rho}-\chi\bold G\hat{\boldsymbol \rho}\right),
\end{equation}
where $\bold G=\frac{1}{\chi}\bold M^{-1}\langle \bold E\rangle$.

\begin{remark}
If one applies the gPC-SG formulation for the limiting Keller-Segel equation (\ref{2-11}) directly, one gets
\begin{equation}\label{4-8}
\partial_t \tilde{\boldsymbol \rho}=\partial_x\left(D\tilde{\bold M}\partial_x\tilde{\boldsymbol \rho}-\chi \tilde{\bold G}\tilde{\boldsymbol\rho}\right),
\end{equation}
where $\tilde{\bold M}=(\tilde m_{ij})_{1\leq i,j\leq K}$ and $\tilde{\bold G}=(\tilde g_{ij})_{1\leq i,j\leq K}$ are $K\times K$ symmetric matrix with entries
\begin{subequations}\label{4-9}
\begin{align}
&\tilde m_{ij}=\int_{I_z}\frac{1}{\alpha_+(z)}\Phi_i(z)\Phi_j(z)\lambda(z)dz,\\
&\tilde g_{ij}=\int_{I_z}(\partial_xs_N)\Phi_i(z)\Phi_j(z)\lambda(z)dz.
\end{align}
\end{subequations}
Although $\tilde{\bold M}$ is different from $\bold M^{-1}$, proof in \cite{ShuHuJin} shows that $\tilde {\bold M}\partial_x\tilde{\boldsymbol \rho}$ and $\bold M^{-1}\partial_x\hat{\boldsymbol \rho}$ are spectrally close to each other. The same property holds between $\tilde {\bold G}\tilde{\boldsymbol \rho}$ and $\bold G\hat{\boldsymbol \rho}$.
\end{remark}

\section{An efficient sAP Scheme Based on an IMEX-RK Method}
One can apply the relaxation method as in \cite{carrillo2013asymptotic} to the projected system (\ref{4-3}), which falls into the sAP framework proposed in \cite{jin2015asymptotic} . However, the method suffers from the parabolic CFL condition $\Delta t=O((\Delta x)^2)$.

Here we propose an efficient sAP scheme using the idea from \cite{boscarino2013implicit} to get rid of the parabolic CFL condition. By adding and subtracting the term $\mu\bar F(v)\partial_x(D\tilde{\bold M}\partial_x\hat{\boldsymbol \rho}-\chi\tilde{\bold G}\hat{\boldsymbol \rho})$\ in (\ref{4-3}a) and the term $\phi v\partial_x \hat{\bold r}$ in (\ref{4-3}b), we reformulate the problem into an equivalent form:
\begin{subequations}\label{5-1}
\begin{align}
\partial_t \hat{\bold r}=&-v\partial_x\hat{\bold j}-\mu\bar F(v)\partial_x(D\tilde{\bold M}\partial_x\hat{\boldsymbol \rho}-\chi\tilde{\bold G}\hat{\boldsymbol \rho})+\frac{1}{\varepsilon^2}\left(\bar F(v)\bold M \hat{\boldsymbol \rho}+\bold B\hat{\boldsymbol \rho}-\bold M\hat{\bold r}-\bold C\hat{\bold r}\right)+\mu\bar F(v)\partial_x(D\tilde{\bold M}\partial_x\hat{\boldsymbol \rho}-\chi\tilde{\bold G}\hat{\boldsymbol \rho})\nonumber\\
=& f_1(\hat{\bold r},\hat{\bold j})+f_2(\hat{\bold r},\hat{\bold s}),\\
\partial_t\hat{\bold j}=&-\phi v\partial_x\hat{\bold r}-\frac{1}{\varepsilon^2}\left[(1-\varepsilon^2\phi)v\partial_x \hat{\bold r}-\bold E\hat{\boldsymbol \rho}+\bold M\hat{\bold j}+\bold C\hat{\bold j}\right]=g_1(\hat{\bold r})+g_2(\hat{\bold r},\hat{\bold j}),\\
\hat{\bold s}=&-\frac{1}{\pi}\log|x|\ast \hat{\boldsymbol\rho}=h(\hat{\bold r}),
\end{align}
\end{subequations}
where $\bold M,\tilde{\bold M},\bold B, \bold C,\bold E$ and $\tilde{\bold G}$ are the same as defined in (\ref{4-5}) and (\ref{4-9}) and 
\begin{subequations}\label{5-2}
\begin{align}
&f_1(\hat{\bold r},\hat{\bold j})=-v\partial_x\hat{\bold j}-\mu\bar F(v)\partial_x(D\tilde{\bold M}\partial_x\hat{\boldsymbol \rho}-\chi\tilde{\bold G}\hat{\boldsymbol \rho}),\\
&f_2(\hat{\bold r},\hat{\bold s})=\frac{1}{\varepsilon^2}\left(\bar F(v)\bold M \hat{\boldsymbol \rho}+\bold B\hat{\boldsymbol \rho}-\bold M\hat{\bold r}-\bold C\hat{\bold r}\right)+\mu\bar F(v)\partial_x(D\tilde{\bold M}\partial_x\hat{\boldsymbol \rho}-\chi\tilde{\bold G}\hat{\boldsymbol \rho}),\\
&g_1(\hat{\bold r})=-\phi v\partial_x\hat{\bold r},\\
&g_2(\hat{\bold r},\hat{\bold j})=-\frac{1}{\varepsilon^2}\left[(1-\varepsilon^2\phi)v\partial_x \hat{\bold r}-\bold E\hat{\boldsymbol \rho}+\bold M\hat{\bold j}+\bold C\hat{\bold j}\right].
\end{align}
\end{subequations}
Here we choose $\mu=\mu(\varepsilon)$ such that
\begin{equation}\label{5-3}
\begin{aligned}
&\lim_{\varepsilon\to 0}\mu=1,\\
&\mu=0 \ \ \ \mbox{if} \ \ \ \varepsilon=O(1);
\end{aligned}
\end{equation}
and $\phi=\phi(\varepsilon)$ such that
\begin{equation}\label{5-4}
0\leq \phi\leq \frac{1}{\varepsilon^2}.
\end{equation}

The restriction on $\phi$ guarantees the positivity of $\phi(\varepsilon)$ and $(1-\varepsilon^2\phi(\varepsilon))$ so that the problem remains well-posed uniformly in $\varepsilon$. We make the same simple choice of $\phi$ as in \cite{jin2000uniformly}:
\begin{equation}\label{5-5}
\phi(\varepsilon)=\min \left\{1,\frac{1}{\varepsilon^2}\right\}.
\end{equation}

Now we apply an IMEX-RK scheme to system (\ref{5-1}) where $(f_1,g_1)^T$ is evaluated explicitly and $(f_2,g_2)^T$ implicitly, then we obtain
\begin{subequations}\label{5-6}
\begin{align}
&\hat{\bold r}^{n+1}=\hat{\bold r}^n+\Delta t\sum_{k=1}^s\tilde{b}_kf_1(\hat{\bold R}^k,\hat{\bold J}^k)+\Delta t\sum_{k=1}^sb_kf_2(\hat{\bold R}^k,\hat{\bold S}^k),\\
&\hat{\bold j}^{n+1}=\hat{\bold j}^n+\Delta t\sum_{k=1}^s\tilde{b}_kg_1(\hat{\bold R}^k)+\Delta t\sum_{k=1}^sb_kg_2(\hat{\bold R}^k,\hat{\bold J}^k),\\
&\hat{\bold s}^{n+1}=-\frac{1}{\pi}\log|x|\ast\hat{\boldsymbol \rho}^{n+1},
\end{align}
\end{subequations}
where the internal stages are
\begin{subequations}\label{5-7}
\begin{align}
&\hat{\bold R}^k=\hat{\bold r}^n+\Delta t\sum_{l=1}^{k-1}\tilde{a}_{kl}f_1(\hat{\bold R}^l,\hat{\bold J}^l)+\Delta t\sum_{l=1}^ka_{kl}f_2(\hat{\bold R}^l,\hat{\bold S}^l),\\
&\hat{\bold J}^k=\hat{\bold j}^n+\Delta t\sum_{l=1}^{k-1}\tilde{a}_{kl}g_1(\hat{\bold R}^l)+\Delta t\sum_{l=1}^ka_{kl}g_2(\hat{\bold R}^l,\hat{\bold J}^l),\\
&\hat{\bold S}^k=-\frac{1}{\pi}\log|x|\ast\hat{\boldsymbol \Rho}^k.
\end{align}
\end{subequations}

It is obvious that the scheme is characterized by the $s\times s$ matrices 
\begin{equation}\label{5-8}
\tilde A=(\tilde a_{ij}),A=(a_{ij})
\end{equation}
and the vectors $\tilde b, b\in\mathbb R^s$, which can be represented by a double table tableau in the usual Butcher notation
\begin{center}
\begin{TAB}(r,0.5cm,0.5cm)[5pt]{c|c}{c|c}% (rows,min,max)[tabcolsep]{columns}{rows}
$\tilde c$& $\tilde A$\\
& $\tilde b^T$\\
\end{TAB}, \ \ \ \ \ \ \ \ \
\begin{TAB}(r,0.5cm,0.5cm)[5pt]{c|c}{c|c}% (rows,min,max)[tabcolsep]{columns}{rows}
$c$& $A$\\
& $b^T$\\
\end{TAB}.
\end{center}
The coefficients $\tilde c$ and $c$ depend on the explicit part of the scheme:
\begin{equation}\label{5-9}
\tilde c_i=\sum_{j=1}^{i-1}\tilde a_{ij}, \ \ c_i=\sum_{j=1}^ia_{ij}.
\end{equation}

In the literature, there are two main different types of IMEX R-K schemes characterized by the structure of the matrix $A$. We are interested in the IMEX-RK method of type $A$ (see \cite{boscarino2013implicit}) where the matrix $A$ is invertible, so that the implicit parts become more amenable.

As an example, we report the SSP(3,3,2) scheme, which is a second order IMEX scheme we are going to use in Section \ref{Secn}

\begin{equation}
\begin{tabular}{c|c c c} 
0&0&0&0\\ 
1/2&1/2&0&0\\
1&1/2&1/2&0\\
\hline
 &1/3&1/3&1/3
\end{tabular},\qquad
\begin{tabular}{c|c c c} 
1/4&1/4&0&0\\ 
1/4&0&1/4&0\\
1&1/3&1/3&1/3\\
\hline
 &1/3&1/3&1/3
\end{tabular}.
\label{SSP332}
\end{equation}

To obtain $\hat{\bold R}^k$ in each internal stage of (\ref{5-7}), one needs $\hat{\boldsymbol\Rho}^k$ and $\hat{\bold S}^k$ in the implicit part $f_2(\hat{\bold R}^k,\hat{\bold S}^k)$. These quantities can be obtained explicitly by the following procedure. 

Suppose one has computed $\hat{\bold R}^l$ and $\hat{\bold S}^l$ for $l=1,\cdots, k-1$, then according to (\ref{5-7}a)
\begin{equation}\label{5-10}
\begin{aligned}
\hat{\bold R}^k=&\hat{\bold r}^n+\Delta t\sum_{l=1}^{k-1}\left(\tilde{a}_{kl}f_1(\hat{\bold R}^l,\hat{\bold J}^l)+a_{kl}f_2(\hat{\bold R}^l,\hat{\bold S}^l)\right)\\
&+\Delta t a_{kk}\left[\frac{1}{\varepsilon^2}(\bar F(v)\bold M \hat{\boldsymbol \Rho}^k+\bold B^k\hat{\boldsymbol \Rho}^k-\bold M\hat{\bold R}^k-\bold C^k\hat{\bold R}^k)+ \mu\bar F(v)\partial_x(D\tilde{\bold M}\partial_x\hat{\boldsymbol \Rho}^k-\chi\tilde{\bold G}^k\hat{\boldsymbol \Rho}^k)\right]\\
=&\overline{\hat{\bold R}}^{k-1}+\Delta t a_{kk}\left[\frac{1}{\varepsilon^2}(\bar F(v)\bold M \hat{\boldsymbol \Rho}^k+\bold B^k\hat{\boldsymbol \Rho}^k-\bold M\hat{\bold R}^k-\bold C^k\hat{\bold R}^k)+ \mu\bar F(v)\partial_x(D\tilde{\bold M}\partial_x\hat{\boldsymbol \Rho}^k-\chi\tilde{\bold G}^k\hat{\boldsymbol \Rho}^k)\right].\end{aligned}
\end{equation}

Here $\overline{\hat{\bold R}}^{k-1}$ represents all contributions in (\ref{5-10}) from the first $k-1$ stages. Now one takes $\langle \cdot \rangle$ on both sides of (\ref{5-10}) so that $[\bar F(v)\bold M \hat{\boldsymbol \Rho}^k+\bold B^k\hat{\boldsymbol \Rho}^k-\bold M\hat{\bold R}^k-\bold C^k\hat{\bold R}^k]$ is cancelled out on the right hand side and one can approximate $\tilde{\bold G}^k$ by $\tilde{\bold G}^{k-1}$. Now $\hat{\boldsymbol \Rho}^k$ can be obtained from the following diffusion equation in an implicit form:
\begin{equation}\label{5-11}
\hat{\boldsymbol \Rho}^k-\Delta t a_{kk}\mu\partial_x(D\tilde{\bold M}\partial_x\hat{\boldsymbol \Rho}^k-\chi\tilde{\bold G}^{k-1}\hat{\boldsymbol \Rho}^k)=\langle \overline{\hat{\bold R}}^{k-1}\rangle.
\end{equation}
Then it is plugged back to (\ref{5-10}) in order to compute $\hat{\bold R}^k$.

\subsection{The Space Discretization}
Second order accuracy is obtained using an upwind TVD scheme (with minmod slope limiter \cite{leveque1992numerical}) in the explicit transport part and center difference for other second derivatives. During each internal stage (\ref{5-7}),
\begin{subequations}\label{5-12}
\begin{align}
\hat{\bold R}^k_i=&\hat{\bold r}^n_i+\Delta t\sum_{l=1}^{k-1}\tilde a_{kl}\left\{-\frac{v}{2\Delta x}(\hat{\bold J}^l_{i+1}-\hat{\bold J}^l_{i-1})+\frac{v\phi^{1/2}}{2\Delta x}(\hat{\bold R}^l_{i+1}-2\hat{\bold R}^l_i+\hat{\bold R}^l_{i-1})\right.-\frac{v\phi^{1/2}}{4}(\boldsymbol\gamma^l_i-\boldsymbol\gamma^l_{i-1}+\boldsymbol\beta^l_{i+1}-\boldsymbol\beta^l_i)\nonumber\\
&-\frac{\mu}{(\Delta x)^2}\bar F (v)D\tilde{\bold M}\left(\hat{\boldsymbol \Rho}^l_{i+1}-2\hat{\boldsymbol \Rho}^l_{i}+\hat{\boldsymbol \Rho}^l_{i-1}\right)\left.+\frac{\mu}{2\Delta x}\bar F(v)\chi\left(\tilde{\bold G}_{i+1}^l\hat{\boldsymbol \Rho}_{i+1}^l-\tilde{\bold G}_{i-1}^l\hat{\boldsymbol \Rho}_{i-1}^l\right)\right\}\nonumber\\
&+\Delta t \sum_{l=1}^ka_{kl}\left\{\frac{1}{\varepsilon^2}\left(\bar F(v)\bold M \hat{\boldsymbol \Rho}^l_i+\bold B_i^l\hat{\boldsymbol \Rho}_i^l-\bold M\hat{\bold R}_i^l-\bold C_i^l\hat{\bold R}_i^l\right)\right.\nonumber\\
&+\frac{\mu}{(\Delta x)^2}\bar F(v)D\tilde{\bold M}\left(\hat{\boldsymbol \Rho}^l_{i+1}-2\hat{\boldsymbol \Rho}^l_{i}+\hat{\boldsymbol \Rho}^l_{i-1}\right)\left.-\frac{\mu}{2\Delta x}\bar F(v)\chi\left(\tilde{\bold G}_{i+1}^l\hat{\boldsymbol \Rho}_{i+1}^l-\tilde{\bold G}_{i-1}^l\hat{\boldsymbol \Rho}_{i-1}^l\right)\right\},\\
\hat{\bold J}^k_i=&\hat{\bold j}^n_i+\Delta t\sum_{l=1}^{k-1}\tilde a_{kl}\left\{-\frac{v\phi}{2\Delta x}(\hat{\bold R}^l_{i+1}-\hat{\bold R}^l_{i-1})+\frac{v\phi^{1/2}}{2\Delta x}(\hat{\bold J}^l_{i+1}-2\hat{\bold J}^l_i+\hat{\bold J}^l_{i-1})-\frac{v\phi}{4}(\boldsymbol\gamma^l_i-\boldsymbol\gamma^l_{i-1}-\boldsymbol\beta^l_{i+1}+\boldsymbol\beta^l_i)\right\}\nonumber\\
&-\Delta t\sum_{l=1}^ka_{kl}\frac{1}{\varepsilon^2}\left\{(1-\varepsilon^2\phi)v\frac{\hat{\bold R}_{i+1}^l-\hat{\bold R}_{i-1}^l}{2\Delta x}-\bold E_i^l\hat{\boldsymbol \Rho}_i^l+\bold M\hat{\bold J}_i^l+\bold C_i^l\hat{\bold J}_i^l\right\},
\end{align}
\end{subequations}
where
\begin{subequations}\label{5-13}
\begin{align}
\boldsymbol \gamma^l_i=&\frac{1}{\Delta x}\text{minmod}\left(\hat{\bold R}^l_{i+1}+\phi^{-1/2}\hat{\bold J}^l_{i+1}-\hat{\bold R}^l_i-\phi^{-1/2}\hat{\bold J}^l_i,\right. \\
&\left.\hat{\bold R}^l_{i}+\phi^{-1/2}\hat{\bold J}^l_{i}-\hat{\bold R}^l_{i-1}-\phi^{-1/2}\hat{\bold J}^l_{i-1}\right),\\
\boldsymbol \beta^l_i=&\frac{1}{\Delta x}\text{minmod}\left(\hat{\bold R}^l_{i+1}-\phi^{-1/2}\hat{\bold J}^l_{i+1}-\hat{\bold R}^l_i+\phi^{-1/2}\hat{\bold J}^l_i,\right.\\
&\left.\hat{\bold R}^l_{i}-\phi^{-1/2}\hat{\bold J}^l_{i}-\hat{\bold R}^l_{i-1}+\phi^{-1/2}\hat{\bold J}^l_{i-1}\right).
\end{align}
\end{subequations}

Since $\hat{\boldsymbol \Rho}^k$ can be obtained explicitly by (\ref{5-11}), we can fully discretize $\hat{\boldsymbol \Rho}^k_i$ as following:
\begin{equation}\label{5-14}
\hat{\boldsymbol \Rho}_i^k-\Delta t a_{kk}\frac{\mu}{(\Delta x)^2}\left[D\tilde{\bold M}(\hat{\boldsymbol \Rho}_{i-1}^k-2\hat{\boldsymbol \Rho}_{i}^k+\hat{\boldsymbol \Rho}_{i+1}^k)-\chi\left(\tilde{\bold G}_{i+\frac{1}{2}}^{k-1}(\hat{\boldsymbol \Rho}_{i+1}^k-\hat{\boldsymbol \Rho}_{i}^k)-\tilde{\bold G}_{i-\frac{1}{2}}^{k-1}(\hat{\boldsymbol \Rho}_{i}^k-\hat{\boldsymbol \Rho}_{i-1}^k)\right)\right]=\langle \overline{\hat{\bold R}}_i^{k-1}\rangle.
\end{equation}

Then using (\ref{5-14}), the fully discretized $\hat{\bold R}^k_i$ is obtained and subsequently $\hat{\bold J}^k_i$ from the following:
\begin{subequations}\label{5-15}
\begin{align}
&\left(\bold I+\frac{a_{kk}\Delta t}{\varepsilon^2}(\bold M+\bold C_i^k)\right)\hat{\bold R}^k_i\nonumber\\
=&\hat{\bold r}^n_i+\Delta t\sum_{l=1}^{k-1}\tilde a_{kl}\left\{-\frac{v}{2\Delta x}(\hat{\bold J}^l_{i+1}-\hat{\bold J}^l_{i-1})\right.+\frac{v\phi^{1/2}}{2\Delta x}(\hat{\bold R}^l_{i+1}-2\hat{\bold R}^l_i+\hat{\bold R}^l_{i-1})-\frac{v\phi^{1/2}}{4}(\boldsymbol\gamma^l_i-\boldsymbol\gamma^l_{i-1}+\boldsymbol\beta^l_{i+1}-\boldsymbol\beta^l_i)\nonumber\\
&-\frac{\mu}{(\Delta x)^2}\bar F (v)D\tilde{\bold M}\left(\hat{\boldsymbol \Rho}^l_{i+1}-2\hat{\boldsymbol \Rho}^l_{i}+\hat{\boldsymbol \Rho}^l_{i-1}\right)\left.+\frac{\mu}{2\Delta x}\bar F(v)\chi\left(\tilde{\bold G}_{i+1}^l\hat{\boldsymbol \Rho}_{i+1}^l-\tilde{\bold G}_{i-1}^l\hat{\boldsymbol \Rho}_{i-1}^l\right)\right\}\nonumber\\
&+\Delta t \sum_{l=1}^{k-1}a_{kl}\left\{\frac{1}{\varepsilon^2}\left[\bar F(v)\bold M \hat{\boldsymbol \Rho}^l_i+\bold B_i^l\hat{\boldsymbol \Rho}_i^l-\bold M\hat{\bold R}_i^l-\bold C_i^l\hat{\bold R}_i^l\right]\right.\left.+\frac{\mu}{(\Delta x)^2}\bar F(v)D\tilde{\bold M}\left(\hat{\boldsymbol \Rho}^l_{i+1}-2\hat{\boldsymbol \Rho}^l_{i}+\hat{\boldsymbol \Rho}^l_{i-1}\right)\right.\nonumber\\
&\left.-\frac{\mu}{2\Delta x}\bar F(v)\chi\left(\tilde{\bold G}_{i+1}^l\hat{\boldsymbol \Rho}_{i+1}^l-\tilde{\bold G}_{i-1}^l\hat{\boldsymbol \Rho}_{i-1}^l\right)\right\}+\Delta t a_{kk}\left\{\frac{1}{\varepsilon^2}\left[\bar F(v)\bold M \hat{\boldsymbol \Rho}^k_i+\bold B_i^k\hat{\boldsymbol \Rho}_i^k\right]\right.\nonumber\\
&\left.+\frac{\mu}{(\Delta x)^2}\bar F(v)D\tilde{\bold M}\left(\hat{\boldsymbol \Rho}^k_{i+1}-2\hat{\boldsymbol \Rho}^k_{i}+\hat{\boldsymbol \Rho}^k_{i-1}\right)\right.\left.-\frac{\mu}{2\Delta x}\bar F(v)\chi\left(\tilde{\bold G}_{i+1}^k\hat{\boldsymbol \Rho}_{i+1}^k-\tilde{\bold G}_{i-1}^k\hat{\boldsymbol \Rho}_{i-1}^k\right)\right\},\\
&\left(1+\frac{a_{kk}\Delta t}{\varepsilon^2}(\bold M +\bold C_i^k)\right)\hat{\bold J}^k_i\nonumber\\
=&\hat{\bold j}^n_i+\Delta t\sum_{l=1}^{k-1}\tilde a_{kl}\left\{-\frac{v\phi}{2\Delta x}(\hat{\bold R}^l_{i+1}-\hat{\bold R}^l_{i-1})\right.+\frac{v\phi^{1/2}}{2\Delta x}(\hat{\bold J}^l_{i+1}-2\hat{\bold J}^l_i+\hat{\bold J}^l_{i-1})\left.-\frac{v\phi}{4}(\boldsymbol\gamma^l_i-\boldsymbol\gamma^l_{i-1}+\boldsymbol\beta^l_{i+1}-\boldsymbol\beta^l_i)\right\}\nonumber\\
&-\Delta t\sum_{l=1}^{k-1}a_{kl}\frac{1}{\varepsilon^2}\left\{(1-\varepsilon^2\phi)v\frac{\hat{\bold R}_{i+1}^l-\hat{\bold R}_{i-1}^l}{2\Delta x}-\bold E_i^l\hat{\boldsymbol \Rho}_i^l+\bold M \hat{\bold J}_i^l+\bold C_i^l\hat{\bold J}_i^l\right\}\nonumber\\
&-\Delta ta_{kk}\frac{1}{\varepsilon^2}\left\{(1-\varepsilon^2\phi)v\frac{\hat{\bold R}_{i+1}^k-\hat{\bold R}_{i-1}^k}{2\Delta x}-\bold E_i^k\hat{\boldsymbol \Rho}_i^k\right\},
\end{align}
\end{subequations}
In the above $\left(1+\frac{a_{kk}\Delta t}{\varepsilon^2}(\bold M +\bold C_i^k)\right)$ is symmetric positive definite, thus invertible. After calculating all $\hat{\bold R}^k_i$ and $\hat{\bold J}^k_i$ for $k=1,\cdots, s$, we can update $\hat{\bold r}^{n+1}_i$ and $\hat{\bold j}^{n+1}_i$ in (\ref{5-6}),
\begin{subequations}\label{5-16}
\begin{align}
\hat{\bold r}^{n+1}_i=&\hat{\bold r}^n_i+\Delta t\sum_{k=1}^{s}\tilde b_{k}\left\{-\frac{v}{2\Delta x}(\hat{\bold J}^k_{i+1}-\hat{\bold J}^k_{i-1})+\frac{v\phi^{1/2}}{2\Delta x}(\hat{\bold R}^k_{i+1}-2\hat{\bold R}^k_i+\hat{\bold R}^k_{i-1})\right.\nonumber\\
&-\frac{v\phi^{1/2}}{4}(\boldsymbol\gamma^k_i-\boldsymbol\gamma^k_{i-1}+\boldsymbol\beta^k_{i+1}-\boldsymbol\beta^k_i)-\frac{\mu}{(\Delta x)^2}\bar F (v)D\tilde{\bold M}\left(\hat{\boldsymbol \Rho}^k_{i+1}-2\hat{\boldsymbol \Rho}^k_{i}+\hat{\boldsymbol \Rho}^k_{i-1}\right)\nonumber\\
&\left.+\frac{\mu}{2\Delta x}\bar F(v)\chi\left(\tilde{\bold G}_{i+1}^k\hat{\boldsymbol \Rho}_{i+1}^k-\tilde{\bold G}_{i-1}^k\hat{\boldsymbol \Rho}_{i-1}^k\right)\right\}\nonumber\\
&+\Delta t \sum_{k=1}^sb_{k}\left\{\frac{1}{\varepsilon^2}\left[\bar F(v)\bold M \hat{\boldsymbol \Rho}^k_i+\bold B_i^k\hat{\boldsymbol \Rho}_i^k-\bold M\hat{\bold R}_i^k-\bold C_i^k\hat{\bold R}_i^k\right]\right.\nonumber\\
&\left.+\frac{\mu}{(\Delta x)^2}\bar F(v)D\tilde{\bold M}\left(\hat{\boldsymbol \Rho}^k_{i+1}-2\hat{\boldsymbol \Rho}^k_{i}+\hat{\boldsymbol \Rho}^k_{i-1}\right)-\frac{\mu}{2\Delta x}\bar F(v)\chi\left(\tilde{\bold G}_{i+1}^k\hat{\boldsymbol \Rho}_{i+1}^k-\tilde{\bold G}_{i-1}^k\hat{\boldsymbol \Rho}_{i-1}^k\right)\right\},\\
\hat{\bold j}^{n+1}_i=&\hat{\bold j}^n_i+\Delta t\sum_{k=1}^{s}\tilde b_{k}\left\{-\frac{v\phi}{2\Delta x}(\hat{\bold R}^k_{i+1}-\hat{\bold R}^k_{i-1})+\frac{v\phi^{1/2}}{2\Delta x}(\hat{\bold J}^k_{i+1}-2\hat{\bold J}^k_i+\hat{\bold J}^k_{i-1})\right.\nonumber\\
&\left.-\frac{v\phi}{4}(\boldsymbol\gamma^k_i-\boldsymbol\gamma^k_{i-1}-\boldsymbol\beta^k_{i+1}+\boldsymbol\beta^k_i)\right\}-\Delta t\sum_{k=1}^sb_{k}\frac{1}{\varepsilon^2}\left\{(1-\varepsilon^2\phi)v\frac{\hat{\bold R}_{i+1}^k-\hat{\bold R}_{i-1}^k}{2\Delta x}\right.\nonumber\\
&\left.-\bold E_i^k\hat{\boldsymbol \Rho}_i^k+\bold M\hat{\bold J}_i^k+\bold C_i^k\hat{\bold J}_i^k\right\},
\end{align}
\end{subequations}
where $\boldsymbol \gamma^k_i$ and $\boldsymbol \beta^k_i$ are defined the same as in (\ref{5-13}).

Following \cite{boscarino2013implicit} we choose 
\begin{equation}\label{5-17}
\mu=\exp(-\varepsilon^2/\Delta x).
\end{equation}
Thus, for large value of $\varepsilon$, (e.g., $\varepsilon=1$), we could avoid the loss of accuracy caused by adding and subtracting the penalty term; for very small value of $\varepsilon$, (e.g., $\varepsilon\to 0$), $\mu\to 1$.

\begin{remark}
The full discrete scheme is obtained using the Gauss-Legendre quadrature nodes for the velocity discretization. Finally, to get the boundary conditions for $\hat{\bold r},\hat{\bold j}$ and $\hat{\bold s}$, we refer to \cite{jin2000uniformly} for details.
\end{remark}

\subsection{The sAP property}
Denote
\begin{subequations}\label{5-18}
\begin{align}
f_1(\hat{\bold R}^l_i,\hat{\bold J}^l_i)=&-\frac{v}{2\Delta x}(\hat{\bold J}^l_{i+1}-\hat{\bold J}^l_{i-1})+\frac{v\phi^{1/2}}{2\Delta x}(\hat{\bold R}^l_{i+1}-2\hat{\bold R}^l_i+\hat{\bold R}^l_{i-1})\nonumber\\
&-\frac{v\phi^{1/2}}{4}(\boldsymbol\gamma^l_i-\boldsymbol\gamma^l_{i-1}+\boldsymbol\beta^l_{i+1}-\boldsymbol\beta^l_i)-\frac{\mu}{(\Delta x)^2}\bar F (v)D\tilde{\bold M}\left(\hat{\boldsymbol \Rho}^l_{i+1}-2\hat{\boldsymbol \Rho}^l_{i}+\hat{\boldsymbol \Rho}^l_{i-1}\right)\nonumber\\
&+\frac{\mu}{2\Delta x}\bar F(v)\chi\left(\tilde{\bold G}_{i+1}^l\hat{\boldsymbol \Rho}_{i+1}^l-\tilde{\bold G}_{i-1}^l\hat{\boldsymbol \Rho}_{i-1}^l\right),\\
f_2(\hat{\bold R}^l_i)=&\frac{1}{\varepsilon^2}\left[\bar F(v)\bold M \hat{\boldsymbol \Rho}^l_i+\bold B_i^l\hat{\boldsymbol \Rho}_i^l-\bold M\hat{\bold R}_i^l-\bold C_i^l\hat{\bold R}_i^l\right]\nonumber\\
&+\frac{\mu}{(\Delta x)^2}\bar F(v)D\tilde{\bold M}\left(\hat{\boldsymbol \Rho}^l_{i+1}-2\hat{\boldsymbol \Rho}^l_{i}+\hat{\boldsymbol \Rho}^l_{i-1}\right)-\frac{\mu}{2\Delta x}\bar F(v)\chi\left(\tilde{\bold G}_{i+1}^l\hat{\boldsymbol \Rho}_{i+1}^l-\tilde{\bold G}_{i-1}^l\hat{\boldsymbol \Rho}_{i-1}^l\right),\\
g_1(\hat{\bold R}^l_i)=&-\frac{v\phi}{2\Delta x}(\hat{\bold R}^l_{i+1}-\hat{\bold R}^l_{i-1})+\frac{v\phi^{1/2}}{2\Delta x}(\hat{\bold J}^l_{i+1}-2\hat{\bold J}^l_i+\hat{\bold J}^l_{i-1})-\frac{v\phi}{4}(\boldsymbol\gamma^l_i-\boldsymbol\gamma^l_{i-1}-\boldsymbol\beta^l_{i+1}+\boldsymbol\beta^l_i),\\
g_2(\hat{\bold R}^l_i,\hat{\bold J}^l_i)=&\frac{1}{\varepsilon^2}\left[(1-\varepsilon^2\phi)v\frac{\hat{\bold R}_{i+1}^l-\hat{\bold R}_{i-1}^l}{2\Delta x}-\bold E_i^l\hat{\boldsymbol \Rho}_i^l+\bold M\hat{\bold J}_i^l+\bold C_i^l\hat{\bold J}_i^l\right].
\end{align}
\end{subequations}

From (\ref{5-15}) we have
\begin{subequations}\label{5-19}
\begin{align}
\begin{pmatrix}
\hat{\bold R}_i^1\\
\hat{\bold R}_i^2\\
\vdots\\
\hat{\bold R}_i^s
\end{pmatrix}=&
\begin{pmatrix}
\hat{\bold r}_i^n\\
\hat{\bold r}_i^n\\
\vdots\\
\hat{\bold r}_i^n
\end{pmatrix}+\Delta t\begin{pmatrix}
0\\
\tilde a_{21}f_1(\hat{\bold R}_i^1,\hat{\bold J}_i^1)\\
\vdots\\
\sum_{l=1}^{s-1}\tilde a_{sl}f_1(\hat{\bold R}_i^{l},\hat{\bold J}_i^{l})
\end{pmatrix}+\Delta t\bold A
\begin{pmatrix}
f_2(\hat{\bold R}_i^1)\\
f_2(\hat{\bold R}_i^2)\\
\vdots\\
f_2(\hat{\bold R}_i^{s})
\end{pmatrix},\\
\begin{pmatrix}
\hat{\bold J}_i^1\\
\hat{\bold J}_i^2\\
\vdots\\
\hat{\bold J}_i^s
\end{pmatrix}=&
\begin{pmatrix}
\hat{\bold j}_i^n\\
\hat{\bold j}_i^n\\
\vdots\\
\hat{\bold j}_i^n
\end{pmatrix}+\Delta t\begin{pmatrix}
0\\
\tilde a_{21}g_1(\hat{\bold R}_i^1)\\
\vdots\\
\sum_{l=1}^{s-1}\tilde a_{sl}g_1(\hat{\bold R}_i^{l})
\end{pmatrix}+\Delta t\bold A\begin{pmatrix}
g_2(\hat{\bold R}_i^1,\hat{\bold J}_i^1)\\
g_2(\hat{\bold R}_i^2,\hat{\bold J}_i^2)\\
\vdots\\
g_2(\hat{\bold R}_i^{s},\hat{\bold J}_i^s)
\end{pmatrix},
\end{align}
\end{subequations}
where 
\begin{equation}\label{5-20}
\bold A_{K(i-1)+1:Ki,K(j-1)+1:Kj}=A_{i,j}\bold I_{K\times K}, \ \ \ \bold I_{K\times K} \ \ \text{is} \ K\times K \ \text{identity matrix},
\end{equation}
and $A$ is defined in (\ref{5-8}). Denote $\bold W$ as the inverse matrix of $\bold A$, then we obtain from (\ref{5-19})
\begin{subequations}\label{5-21}
\begin{align}
&\Delta t
\begin{pmatrix}
f_2(\hat{\bold R}_i^1)\\
f_2(\hat{\bold R}_i^2)\\
\vdots\\
f_2(\hat{\bold R}_i^{s})
\end{pmatrix}=\bold W\left[\begin{pmatrix}
\hat{\bold R}_i^1\\
\hat{\bold R}_i^2\\
\vdots\\
\hat{\bold R}_i^s
\end{pmatrix}-
\begin{pmatrix}
\hat{\bold r}_i^n\\
\hat{\bold r}_i^n\\
\vdots\\
\hat{\bold r}_i^n
\end{pmatrix}-\Delta t\begin{pmatrix}
0\\
\tilde a_{21}f_1(\hat{\bold R}_i^1,\hat{\bold J}_i^1)\\
\vdots\\
\sum_{l=1}^{s-1}\tilde a_{sl}f_1(\hat{\bold R}_i^{l},\hat{\bold J}_i^{l})
\end{pmatrix}\right],\\
&\Delta t\begin{pmatrix}
g_2(\hat{\bold R}_i^1,\hat{\bold J}_i^1)\\
g_2(\hat{\bold R}_i^2,\hat{\bold J}_i^2)\\
\vdots\\
g_2(\hat{\bold R}_i^{s},\hat{\bold J}_i^s)
\end{pmatrix}=\bold W\left[\begin{pmatrix}
\hat{\bold J}_i^1\\
\hat{\bold J}_i^2\\
\vdots\\
\hat{\bold J}_i^s
\end{pmatrix}-
\begin{pmatrix}
\hat{\bold j}_i^n\\
\hat{\bold j}_i^n\\
\vdots\\
\hat{\bold j}_i^n
\end{pmatrix}-\Delta t\begin{pmatrix}
0\\
\tilde a_{21}g_1(\hat{\bold R}_i^1)\\
\vdots\\
\sum_{l=1}^{s-1}\tilde a_{sl}g_1(\hat{\bold R}_i^{l})
\end{pmatrix}\right].
\end{align}
\end{subequations}

Since $\bold W$ has the same structure as $\bold A$, $\bold W$ should be a lower  triangular matrix with entries
\begin{equation}\label{51}
\bold W_{K(i-1)+1:Ki,K(j-1)+1:Kj}=\omega_{i,j}\bold I_{K\times K},
\end{equation}
where $W=(\omega_{i,j})$ is the inverse of the lower triangular matrix $A$ in (\ref{5-8}).

Then one can rewrite (\ref{5-21}) as 
\begin{subequations}\label{5-23}
\begin{align}
&\Delta tf_2(\hat{\bold R}_i^k)=\sum_{l=1}^k\omega_{kl}\left[\hat{\bold R}_i^l-\hat{\bold r}_i^n-\Delta t\sum_{l=1}^{k-1}\tilde{a}_{kl}f_1(\hat{\bold R}_i^l,\hat{\bold J}_i^l)\right],\\
&\Delta tg_2(\hat{\bold R}_i^k,\hat{\bold J}_i^k)=\sum_{l=1}^k\omega_{kl}\left[\hat{\bold J}_i^l-\hat{\bold j}_i^n-\Delta t\sum_{l=1}^{k-1}\tilde a_{kl}g_1(\hat{\bold R}_i^l)\right].
\end{align}
\end{subequations}
More explicitly,
\begin{subequations}\label{5-24}
\begin{align}
&\Delta t\left\{\frac{1}{\varepsilon^2}\left[\bar F(v)\bold M \hat{\boldsymbol \Rho}^k_i+\bold B_i^k\hat{\boldsymbol \Rho}_i^k-\bold M\hat{\bold R}_i^k-\bold C_i^k\hat{\bold R}_i^k\right]\right.\nonumber\\
&\left.+\frac{\mu}{(\Delta x)^2}\bar F(v)D\tilde{\bold M}\left(\hat{\boldsymbol \Rho}^k_{i+1}-2\hat{\boldsymbol \Rho}^k_{i}+\hat{\boldsymbol \Rho}^k_{i-1}\right)-\frac{\mu}{2\Delta x}\bar F(v)\chi\left(\tilde{\bold G}_{i+1}^k\hat{\boldsymbol \Rho}_{i+1}^k-\tilde{\bold G}_{i-1}^k\hat{\boldsymbol \Rho}_{i-1}^k\right)\right\}\nonumber\\
=&\sum_{l=1}^k\omega_{kl}\left\{\hat{\bold R}_i^l-\hat{\bold r}_i^n-\Delta t\sum_{l=1}^{k-1}\tilde{a}_{kl}\left[-\frac{v}{2\Delta x}(\hat{\bold J}^l_{i+1}-\hat{\bold J}^l_{i-1})+\frac{v\phi^{1/2}}{2\Delta x}(\hat{\bold R}^l_{i+1}-2\hat{\bold R}^l_i+\hat{\bold R}^l_{i-1})\right.\right.\nonumber\\
&-\frac{v\phi^{1/2}}{4}(\boldsymbol\gamma^l_i-\boldsymbol\gamma^l_{i-1}+\boldsymbol\beta^l_{i+1}-\boldsymbol\beta^l_i)-\frac{\mu}{(\Delta x)^2}\bar F (v)D\tilde{\bold M}\left(\hat{\boldsymbol \Rho}^l_{i+1}-2\hat{\boldsymbol \Rho}^l_{i}+\hat{\boldsymbol \Rho}^l_{i-1}\right)\nonumber\\
&\left.+\frac{\mu}{2\Delta x}\bar F(v)\chi\left(\tilde{\bold G}_{i+1}^l\hat{\boldsymbol \Rho}_{i+1}^l-\tilde{\bold G}_{i-1}^l\hat{\boldsymbol \Rho}_{i-1}^l\right)\right\},\\
&\Delta t\left\{\frac{1}{\varepsilon^2}\left[(1-\varepsilon^2\phi)v\frac{\hat{\bold R}_{i+1}^k-\hat{\bold R}_{i-1}^k}{2\Delta x}-\bold E_i^k\hat{\boldsymbol \Rho}_i^k+\bold M\hat{\bold J}_i^k+\bold C_i^k\hat{\bold J}_i^k\right]\right\}\nonumber\\
=&\sum_{l=1}^k\omega_{kl}\left\{\hat{\bold J}_i^l-\hat{\bold j}_i^n-\Delta t \sum_{l=1}^{k-1}\tilde a_{kl}\left[-\frac{v\phi}{2\Delta x}(\hat{\bold R}^l_{i+1}-\hat{\bold R}^l_{i-1})+\frac{v\phi^{1/2}}{2\Delta x}(\hat{\bold J}^l_{i+1}-2\hat{\bold J}^l_i+\hat{\bold J}^l_{i-1})\right.\right.\nonumber\\
&\left.\left.-\frac{v\phi}{4}(\boldsymbol\gamma^l_i-\boldsymbol\gamma^l_{i-1}-\boldsymbol\beta^l_{i+1}+\boldsymbol\beta^l_i)\right]\right\}.
\end{align}
\end{subequations}
Thus, setting $\varepsilon\to 0$, since $\bold M+\bold C_i^k$ is non-singular, one obtains
\begin{subequations}\label{5-25}
\begin{align}
\hat{\bold R}_i^k=&(\bold M+\bold C_i^k)^{-1}(\bar F(v)\bold M+\bold B_i^k)\hat{\boldsymbol \Rho}_i^k=\bar F(v)\hat{\boldsymbol \Rho}_i^k+O(\varepsilon),\\
\hat{\bold J}_i^k=&(\bold M+\bold C_i^k)^{-1}(\bold E_i^k\hat{\boldsymbol \Rho}_i^k-v\frac{\hat{\bold R}_{i+1}^k-\hat{\bold R}_{i-1}^k}{2\Delta x})=\bold M^{-1}\bold E_i^k\hat{\boldsymbol \Rho}_i^k-v\bold M^{-1}\frac{\hat{\bold R}_{i+1}^k-\hat{\bold R}_{i-1}^k}{2\Delta x}+O(\varepsilon).
\end{align}
\end{subequations}

Inserting this back to (\ref{5-16}a) and letting $\varepsilon\to 0$,
\begin{equation}\label{5-26}
\hat{\bold r}_i^{n+1}=\hat{\bold r}_i^n+\Delta t\sum_{k=1}^s\tilde{b}_k\hat f_1(\hat{\bold R}_i^k)+\Delta t\sum_{k=1}^sb_k\hat f_2(\hat{\bold R}_i^k),
\end{equation}
where
\begin{subequations}\label{5-27}
\begin{align}
\hat f_1(\hat{\bold R}_i^k)=&v^2\bar F(v)\frac{\bold M^{-1}}{4(\Delta x)^2}\left(\hat{\bold R}_{i+2}^k-2\hat{\bold R}_{i}^k+\hat{\bold R}_{i-2}^k\right)-\bar{F}(v)D\frac{\tilde{\bold M}}{(\Delta x)^2}\left(\hat{\boldsymbol \Rho}_{i+1}^k-2\hat{\boldsymbol \Rho}_i^k-\hat{\boldsymbol \Rho}_{i-1}^k\right)\nonumber\\
&-\frac{v^2}{4\Delta x}(\bold M^{-1}{\bold E}_{i+1}^k\hat{\boldsymbol \Rho}_{i+1}^k-\bold M^{-1}{\bold E}_{i-1}^k\hat{\boldsymbol \Rho}_{i-1}^k)+\frac{1}{2\Delta x}\bar{F}(v)\chi(\tilde{\bold G}_{i+1}^k\hat{\boldsymbol \Rho}_{i+1}^k-\tilde{\bold G}_{i-1}^k\hat{\boldsymbol \Rho}_{i-1}^k),\\
\hat f_2(\hat{\bold R}_i^k)=&\frac{1}{(\Delta x)^2}\bar F(v)D\tilde{\bold M}\left(\hat{\boldsymbol \Rho}^k_{i+1}-2\hat{\boldsymbol \Rho}^k_{i}+\hat{\boldsymbol \Rho}^k_{i-1}\right)-\frac{1}{2\Delta x}\bar F(v)\chi\left(\tilde{\bold G}_{i+1}^k\hat{\boldsymbol \Rho}_{i+1}^k-\tilde{\bold G}_{i-1}^k\hat{\boldsymbol \Rho}_{i-1}^k\right).
\end{align}
\end{subequations}

Since the difference between $\bold M^{-1}\left(\hat{\boldsymbol \Rho}^k_{i+1}-2\hat{\boldsymbol \Rho}^k_{i}+\hat{\boldsymbol \Rho}^k_{i-1}\right)$ and $\tilde{\bold M}\left(\hat{\boldsymbol \Rho}^k_{i+1}-2\hat{\boldsymbol \Rho}^k_{i}+\hat{\boldsymbol \Rho}^k_{i-1}\right)$ and the difference between $\frac{1}{\chi}\bold M^{-1}\bold E_i^k\hat{\boldsymbol \Rho}^k_{i}$ and $\tilde{\bold G}_i^k\hat{\boldsymbol \Rho}^k_{i}$ are both spectral small, after integrating over $V^+$, $\hat f_1$ goes to $0$ and one gets
\begin{equation}\label{5-28}
\hat{\boldsymbol \rho}_i^{n+1}=\hat{\boldsymbol \rho}_i^n+\Delta t\sum_{k=1}^sb_k\left[\bar F(v)D\tilde{\bold M}\frac{\hat{\boldsymbol \Rho}^k_{i+1}-2\hat{\boldsymbol \Rho}^k_{i}+\hat{\boldsymbol \Rho}^k_{i-1}}{(\Delta x)^2}-\bar F(v)\chi\frac{\tilde{\bold G}_{i+1}^k\hat{\boldsymbol \Rho}_{i+1}^k-\tilde{\bold G}_{i-1}^k\hat{\boldsymbol \Rho}_{i-1}^k}{2\Delta x}\right]+O((\Delta x)^2),
\end{equation}
where
\begin{equation}\label{5-29}
\begin{aligned}
\hat{\boldsymbol \Rho}_i^k=&\hat{\boldsymbol \rho}_i^n+\Delta t\sum_{l=1}^{k-1}a_{kl}\left[\bar F(v)D\tilde{\bold M}\frac{\hat{\boldsymbol \Rho}^l_{i+1}-2\hat{\boldsymbol \Rho}^l_{i}+\hat{\boldsymbol \Rho}^l_{i-1}}{(\Delta x)^2}-\bar F(v)\chi\frac{\tilde{\bold G}_{i+1}^l\hat{\boldsymbol \Rho}_{i+1}^l-\tilde{\bold G}_{i-1}^l\hat{\boldsymbol \Rho}_{i-1}^l}{2\Delta x}\right]\\
&+\Delta ta_{kk}\left[\bar F(v)D\tilde{\bold M}\frac{\hat{\boldsymbol \Rho}^k_{i+1}-2\hat{\boldsymbol \Rho}^k_{i}+\hat{\boldsymbol \Rho}^k_{i-1}}{(\Delta x)^2}-\bar F(v)\chi\frac{\tilde{\bold G}_{i+1}^{k-1}\hat{\boldsymbol \Rho}_{i+1}^k-\tilde{\bold G}_{i-1}^{k-1}\hat{\boldsymbol \Rho}_{i-1}^k}{2\Delta x}\right],
\end{aligned}
\end{equation}
which is an implicit RK scheme for the projected limiting diffusion equation (\ref{4-8}). Thus, the sAP property \cite{jin2015asymptotic} of the efficient IMEX R-K scheme is formally justified.

\section{Numerical Tests}
\label{Secn}
\subsection{The 1D Nonlocal Deterministic Model}
The following numerical tests are carried out with
$$x\in\Omega=[-1,1],\ v\in V=[-1,1],  \ \alpha=1,$$
$$\bar F(v)=\frac{1}{|V|}\mathbf{1}_V:=\left\{
\begin{aligned}
&\frac{1}{|V|}&&\text{if} \ v\in V\\
&0&&\text{otherwise}
\end{aligned}
\right..$$
The critical mass for the limiting Keller-Segel system given by formula (\ref{1-3}) is 
$$M_c=2\pi.$$
The initial conditions are given by
$$\rho_I(x)=Ce^{-80x^2},  \ \ f_I(x,v)=\rho_I(x) F(v),$$
where $C=C(M)$ is a constant determined by the total mass $M$.

For the deterministic case, we compare our results by the second order IMEX-RK method (\ref{SSP332}) (denoted by SSP2 in the figures) with the results by \cite{carrillo2013asymptotic} (denoted by CY in the figures). For both tests, we set $\Delta x=0.005$. In their numerical tests, the CFL condition is 
$$\Delta t=\max \left\{\frac{\varepsilon\Delta x}{2},\frac{\Delta x^2}{2}\right\}.$$
Obviously, when $\varepsilon$ is small, it suffers from the parabolic CFL condition for the diffusive nature of the Keller-Segel system.

For our IMEX-RK method, the choice of $\Delta t$ is given by
$$\Delta t=\lambda \Delta x, \ \lambda=0.02,$$
which is much bigger than ${\Delta x^2}/{2}$.
\subsubsection{A Super-Critical Mass}
It has been shown in \cite{chalub2004kinetic} that the solution of the kinetic system can converge to the Keller-Segel system weakly in a finite time interval $[0, t^*]$, with $t^*<t_b$. Here $t_b$ is the blow up time of the corresponding Keller-Segel system.

For the Super-Critical case, we set 
\begin{equation}\label{6-1}
M=4\pi>M_c=2\pi, \ t=0.003<t_b\approx 0.0039.
\end{equation}
\begin{figure}[H]
\begin{center}
\includegraphics[scale=0.4]{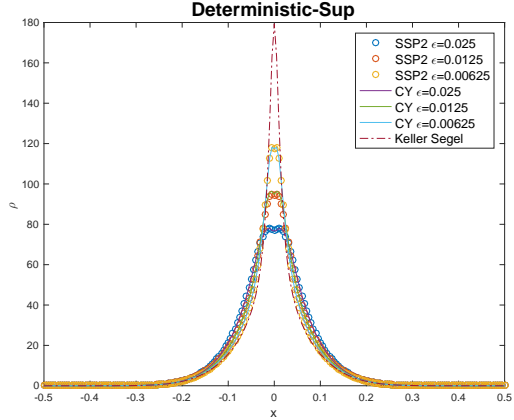}
\end{center}
\caption{The 1D nonlocal deterministic model in the super-critical case. Solid lines are numerical results obtained in \cite{carrillo2013asymptotic} and circles are numerical results obtained by the IMEX-RK method. Dashed line is the numerical solution of the Keller-Segel equations as reference.}
\end{figure}

Figure 1 shows that the solution to the kinetic equation $\rho$ converges to the solution of the Keller-Segel solution $\rho_0$ as $\varepsilon\to 0$ at  time $t=0.003<t_b$. Our IMEX-RK results match very well with the results in \cite{carrillo2013asymptotic}.
\subsubsection{A Sub-Critical Mass}
For the Sub-Critical case, we set 
\begin{equation}\label{6-2}
M=\pi<M_c, \ t=0.1.
\end{equation}
\begin{figure}[H]
\begin{center}
\includegraphics[scale=0.4]{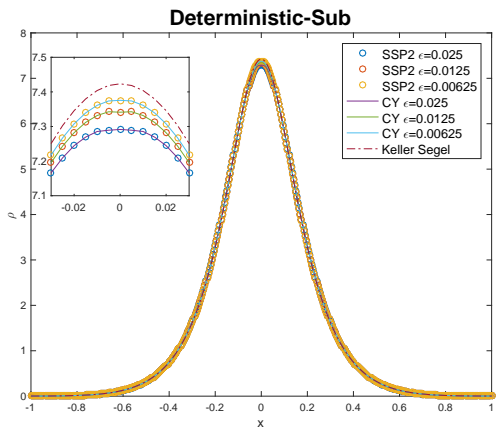}
\end{center}
\caption{The 1D nonlocal deterministic model in the sub-critical case. Solid lines are numerical results obtained in \cite{carrillo2013asymptotic} and circles are numerical results obtained by the IMEX-RK method. Dashed line is the numerical solution of the Keller-Segel equations as reference.}

\end{figure}

Figure 2 shows similar convergence results as the supercritical case for a relatively long time $t=0.1$. Also, good agreements between our new IMEX-RK solutions and the numerical results from \cite{carrillo2013asymptotic} can be observed, even in zoomed in area.

\subsection{The 1D Nonlocal Model with Random Inputs in the Supercritical Case}
Now we let
$$\alpha=1+0.5z, \ z\sim U[-1,1], M=4\pi>\bar M_c\approx2.197\pi.$$
Using the same mesh size as before, we also employ the stochastic collocation method (using 20 quadrature points) as reference solutions. In stochastic collocation, the deterministic solver can be applied directly to a set of selected sample points and then the solution is approximated by interpolation of all sample solutions (see \cite{xiu2010numerical} for a review of stochastic collocation methods). The gPC expansion has been considered only up to $4$th order in our numerical tests. The following are the comparisons of the two methods in mean and standard deviation for the super-critical case with the same initial mass and stopping time in (\ref{6-1}). Given the gPC coefficients $\hat{\rho}_k$ of $\rho$, the mean and standard deviation are calculated as 
$$\mathbb{E}[\rho]\approx\hat{\rho}_1, \ \ \mathbb{S}[\rho]\approx\sqrt{\sum_{k=2}^K\hat{\rho}_k^2}.$$

\subsubsection{The sAP property}
\begin{figure}[H]
\includegraphics[scale=0.4]{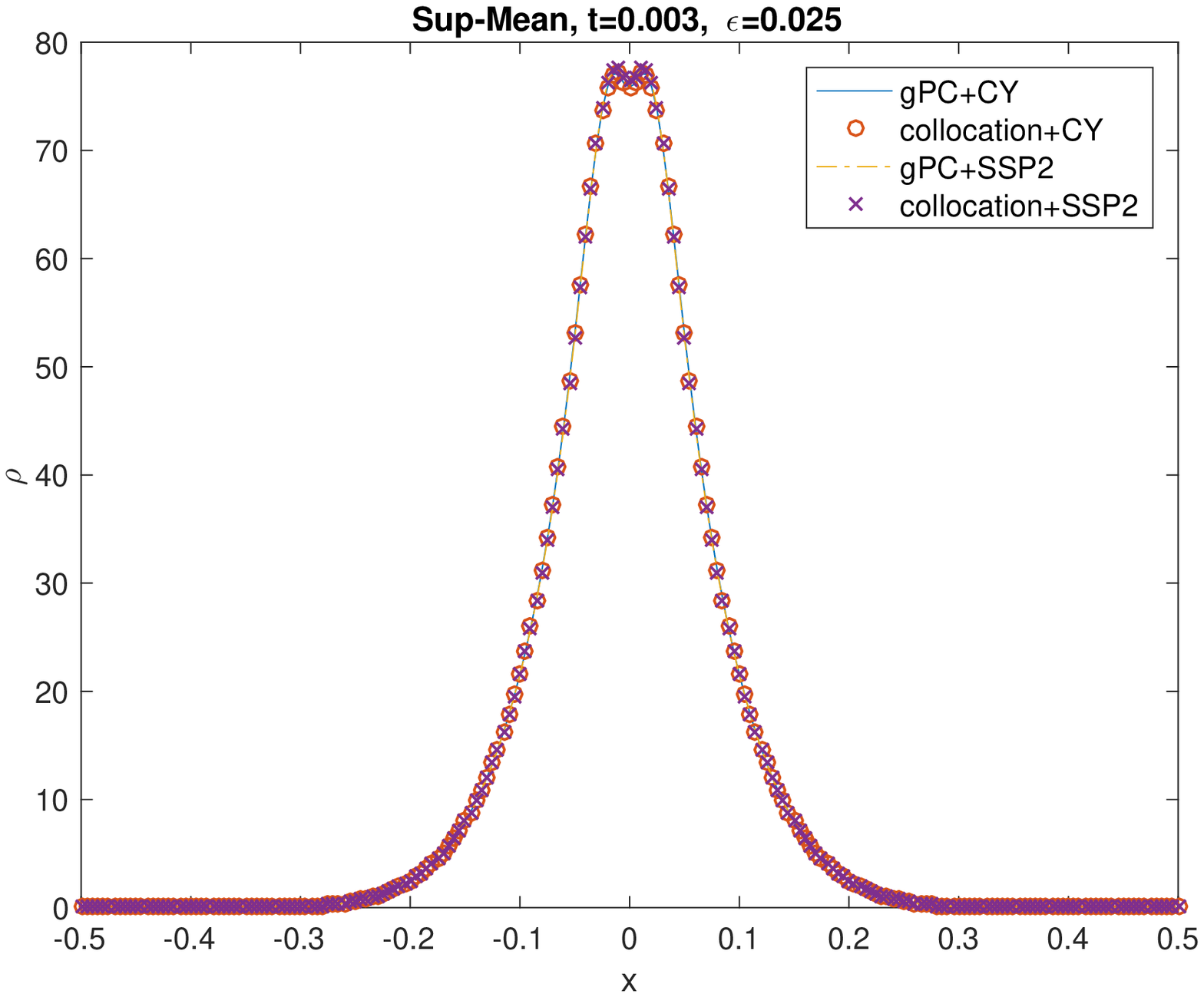}
\includegraphics[scale=0.4]{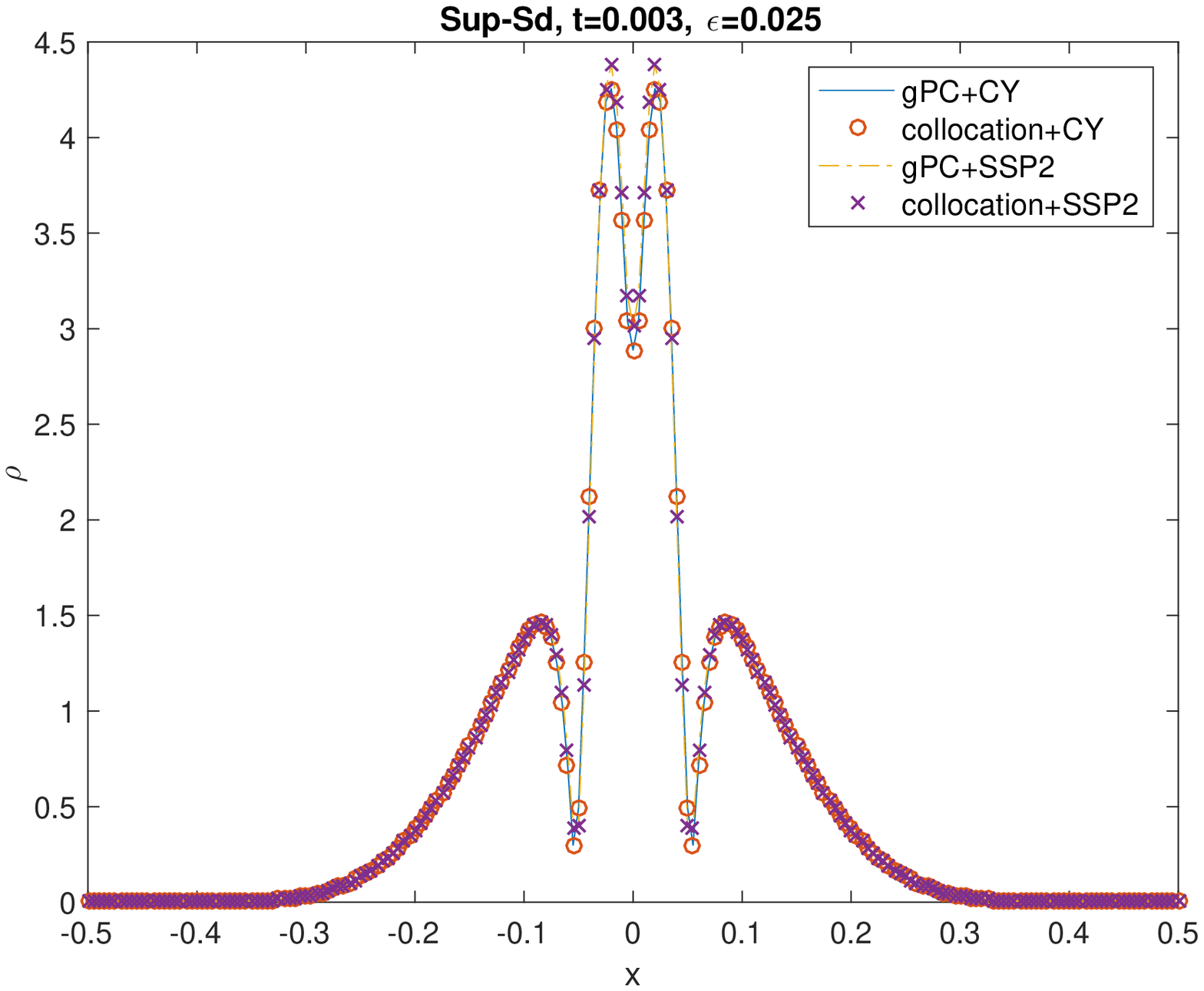}
\includegraphics[scale=0.4]{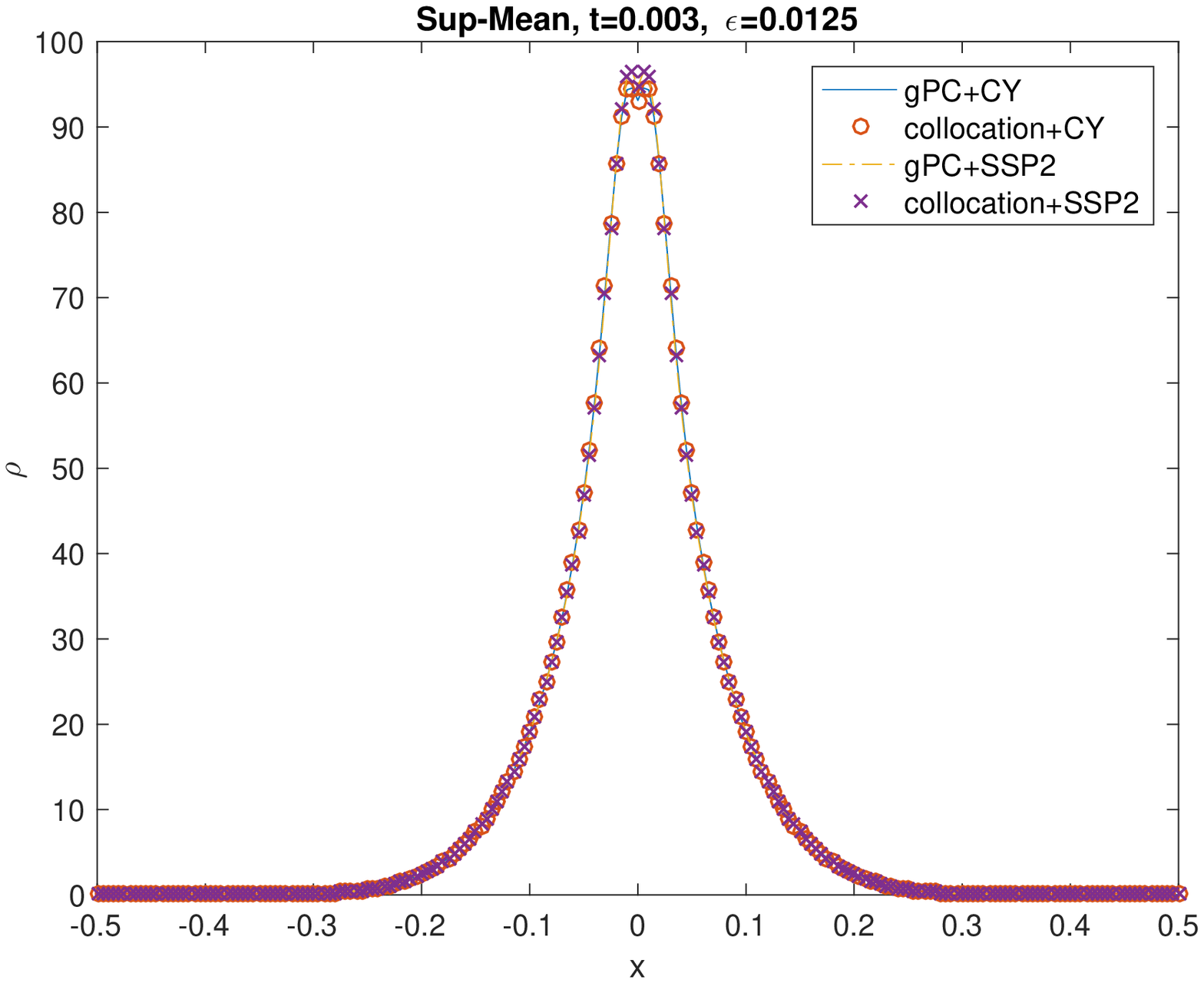}
\includegraphics[scale=0.4]{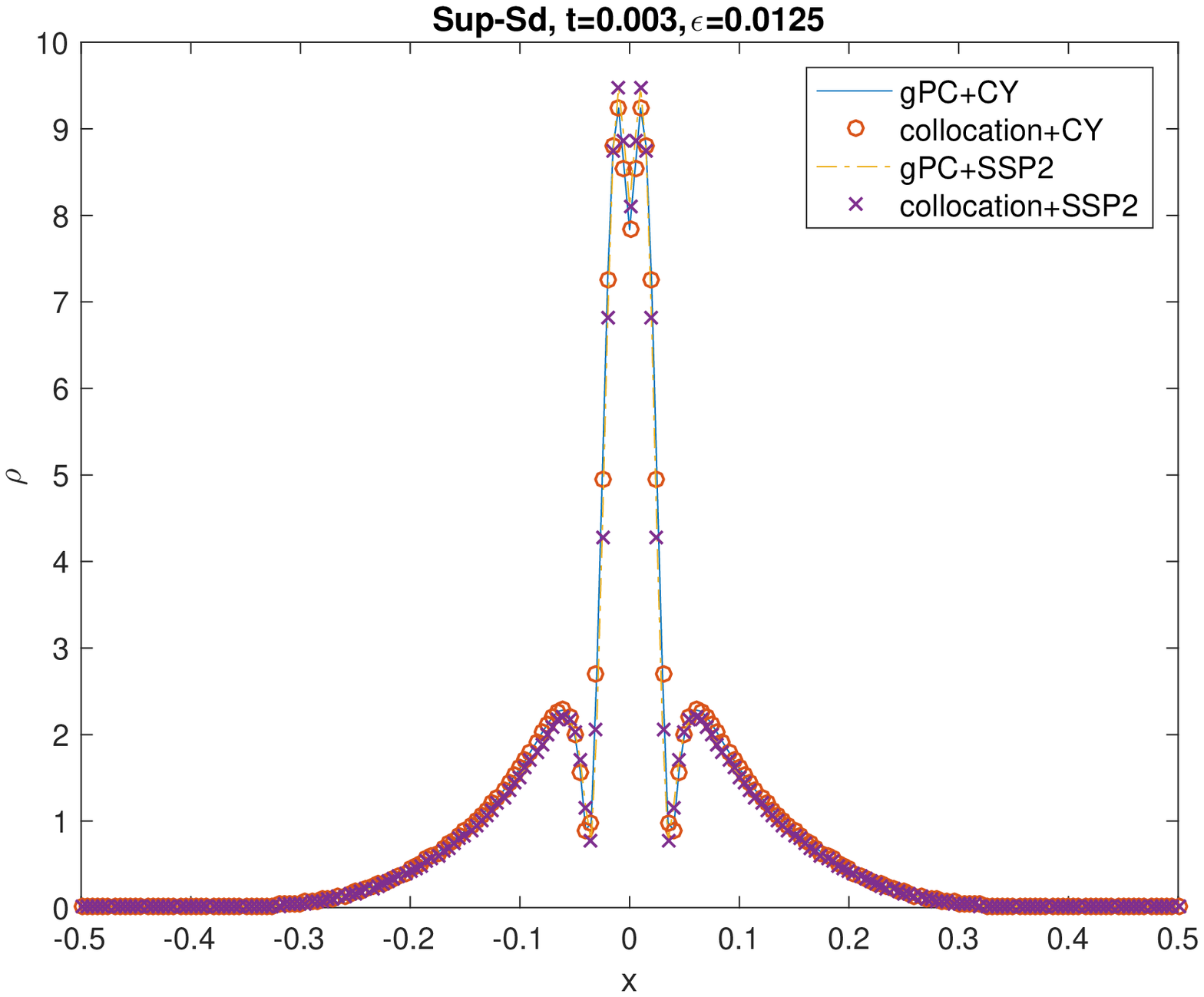}
\includegraphics[scale=0.4]{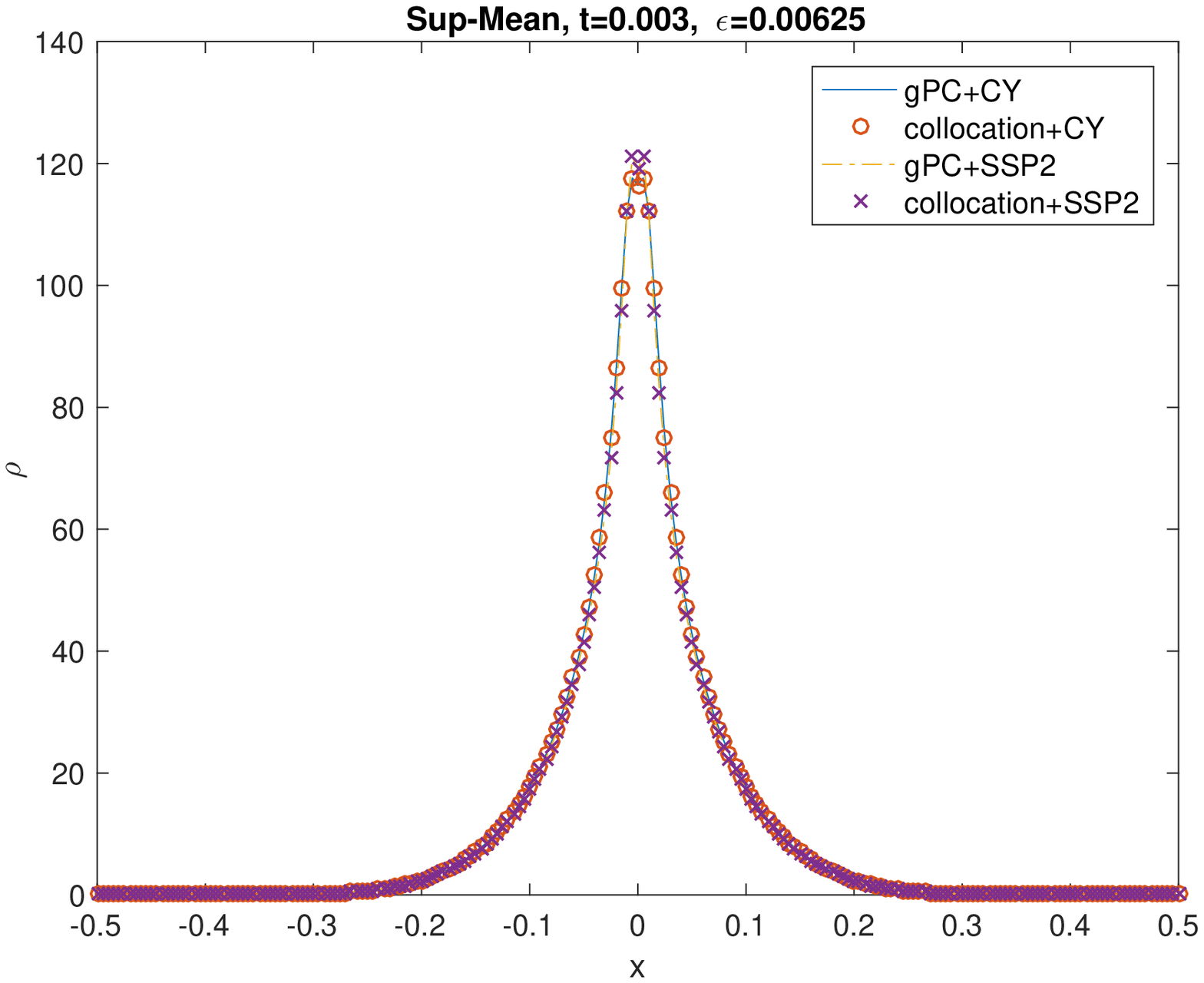}
\includegraphics[scale=0.4]{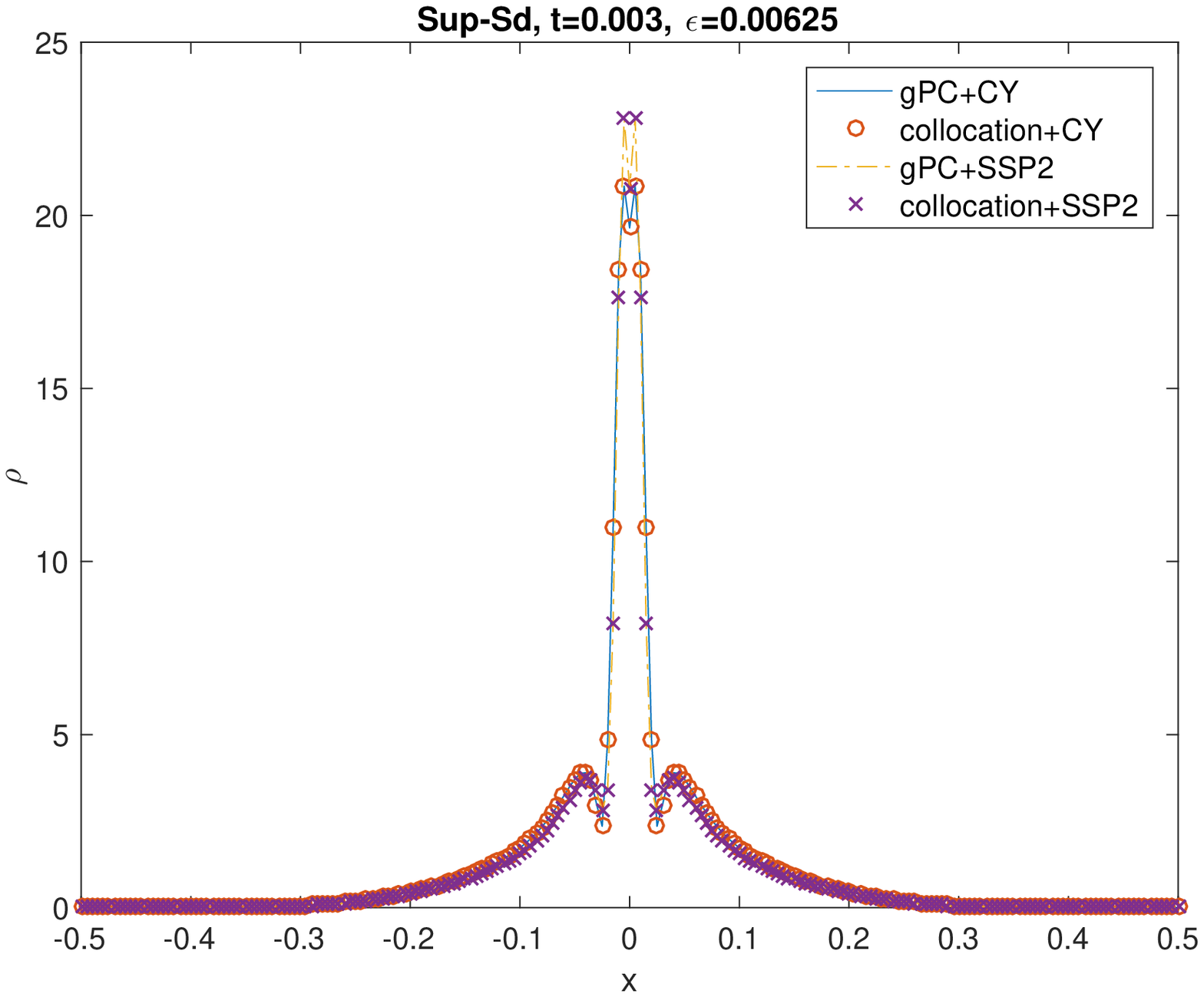}
\caption{The 1D nonlocal random model in the super-critical case. Solid line is obtained by combining the deterministic solver \cite{carrillo2013asymptotic} with the gPC method and circle is obtained by combining the deterministic solver \cite{carrillo2013asymptotic} with the collocation method. Dashed line is obtained by the IMEX-RK using gPC and cross is obtained by the IMEX-RK using collocation. Different values of $\varepsilon$ are tested and the two quantities of interests are mean value (left) and standard deviation (right).}
\end{figure}
\begin{figure}[H]
\includegraphics[scale=0.4]{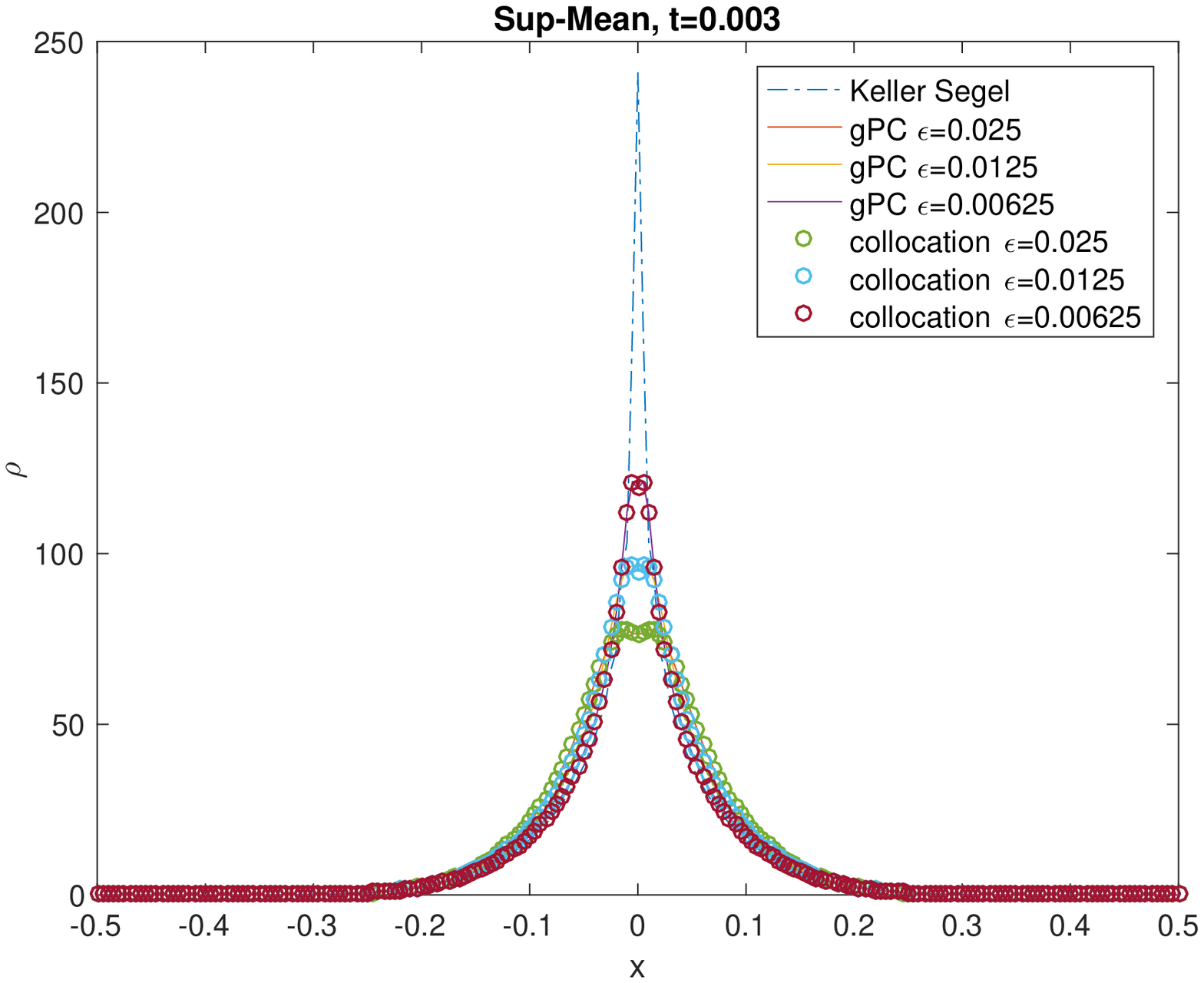}
\includegraphics[scale=0.4]{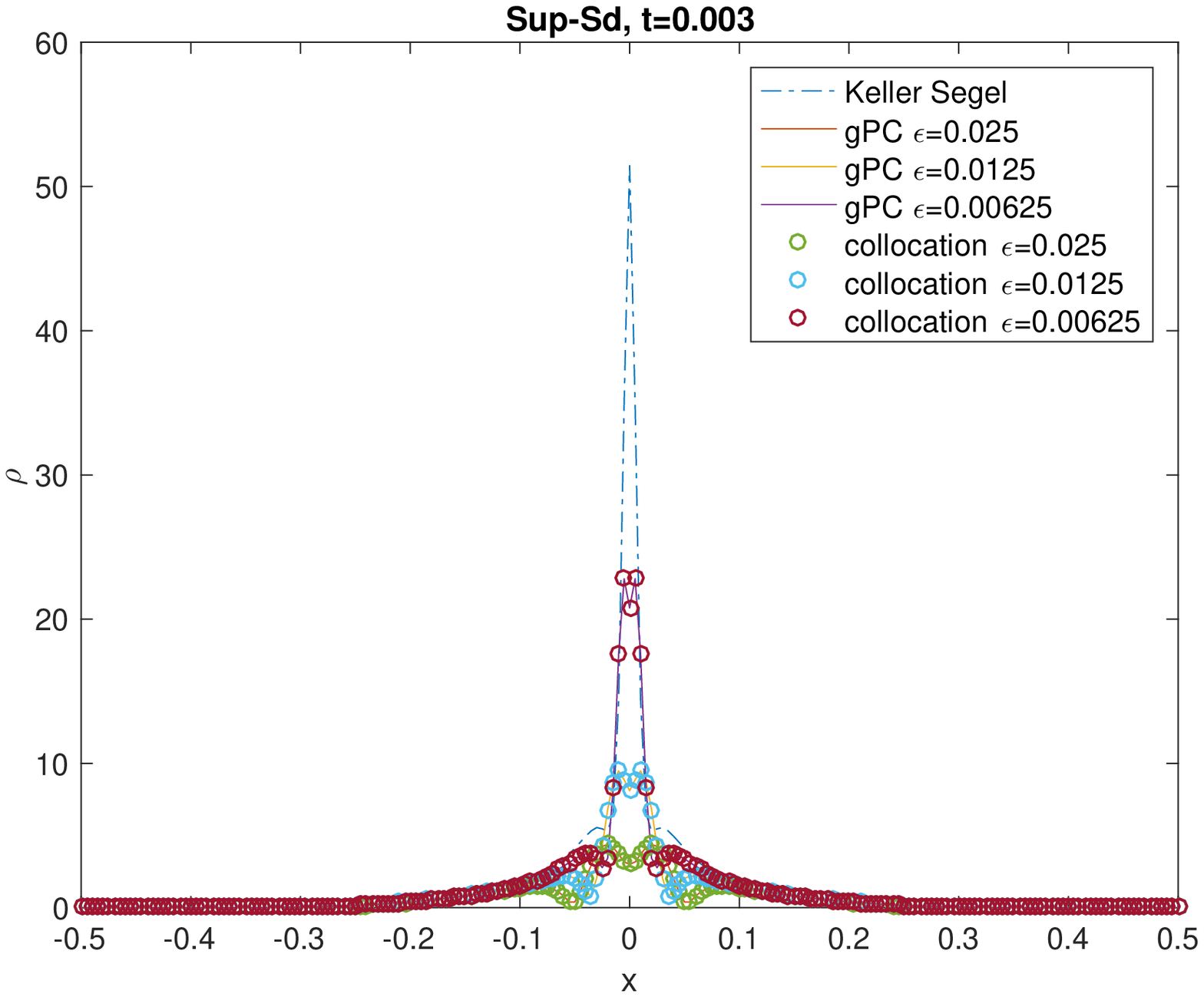}
\caption{The 1D nonlocal random model in the super-critical case. Solid line is obtained by combining the deterministic solver IMEX-RK with the gPC-SG method and circle is obtained by combining the deterministic solver IMEX-RK with the collocation method. Dashed line is the gPC-SG solution of the limiting Keller-Segel equations with uncertainty. Different values of $\varepsilon$ are tested and the two quantities of interests are mean value (left) and standard deviation (right).}
\end{figure}

Figure 3 shows that the IMEX-RK solution agrees well with results of \cite{carrillo2013asymptotic} for all $\varepsilon$ no matter combined with gPC approach or collocation approach to deal with the uncertainty. Small differences between the two methods, especially near the singularity for small $\varepsilon$, are observed due to different orders of accuracy, but the SG solution always matches the collocation solution accurately for the same deterministic solver. Figure 4 shows that the mean and standard deviation of the kinetic chemotaxis solutions both tend to the quantities of the limiting Keller-Segel solution as $\varepsilon\to 0$ for fixed $\Delta t$ and $\Delta x$, which verifies the sAP property.

\subsubsection{Global Existence and Finite Time Blow Up}
\begin{figure}[H]
\includegraphics[scale=0.4]{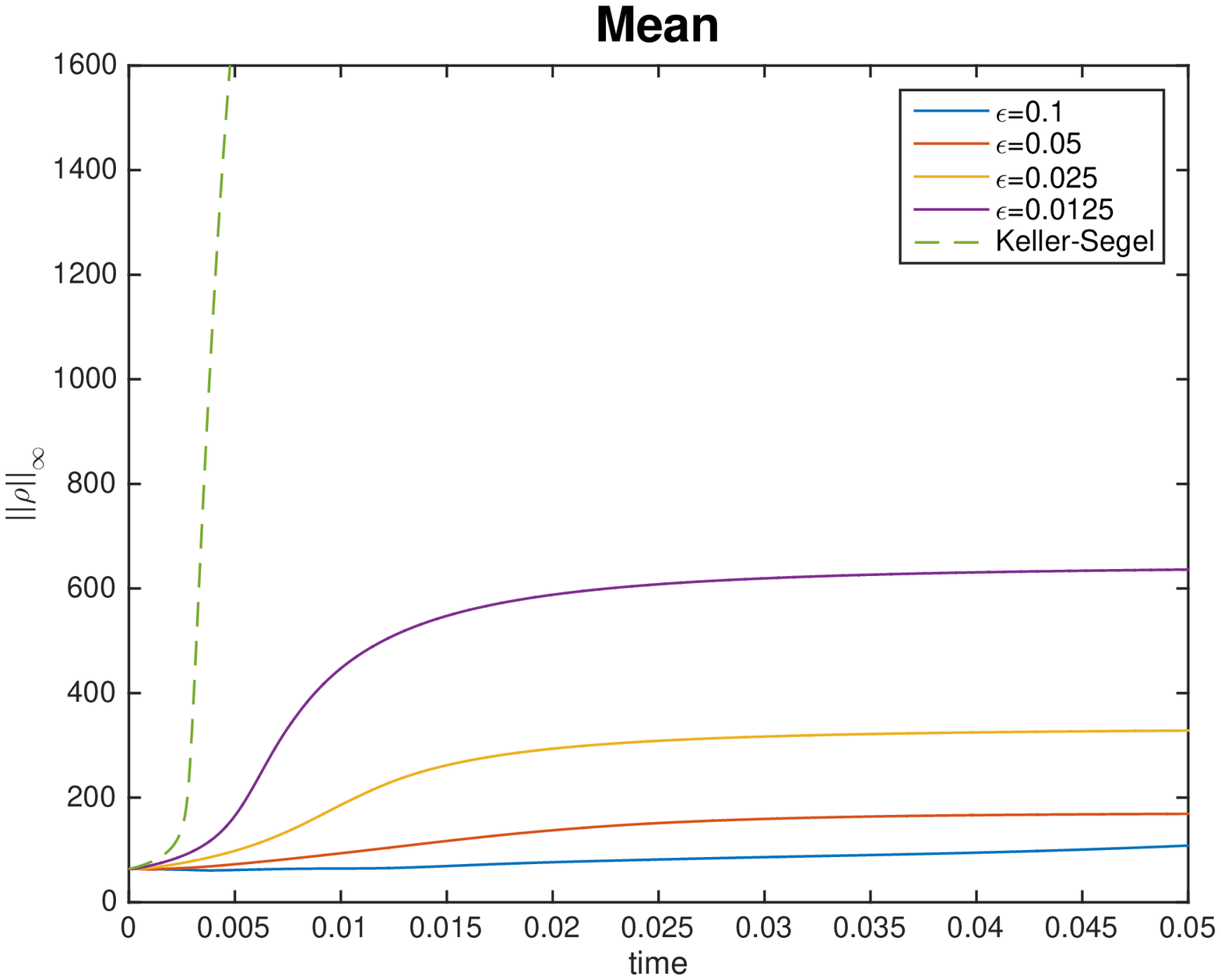}
\includegraphics[scale=0.4]{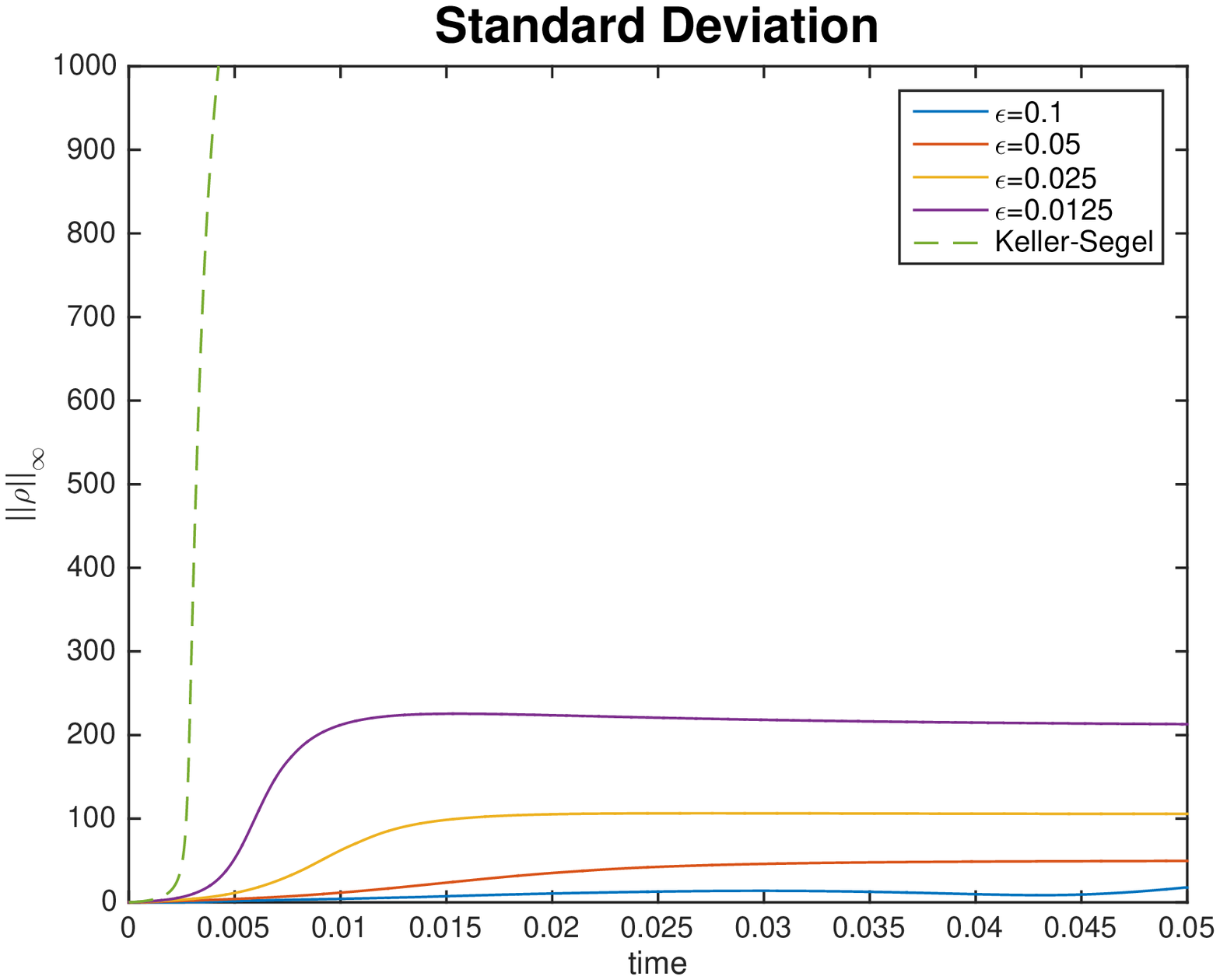}
\caption{The 1D nonlocal random model in the super-critical case. Solid line is obtained by combining the deterministic solver IMEX-RK with the gPC-SG method and dashed line is the gPC-SG solution of the limiting Keller-Segel equations. $\rho$ in infinity norm with different values of $\varepsilon$ are tested and the two quantities of interests are mean value (left) and standard deviation (right).}
\end{figure}
As proved in \cite{chalub2004kinetic}, the solution to the kinetic system (\ref{2-7}) with the nonlocal turning kernel is bounded on $[0,T]$, for any time $T$. However, the Keller-Segel solution will blow up in finite time with a supercritical mass. We examine the mean value and standard deviation of $\|\rho\|_\infty$ for relatively long time ($t\gg t_b$) in Figure 5. The uncertain systems show the same properties as the deterministic ones, e.g. the kinetic systems have global bound in the first and second moments for different $\varepsilon$ while the Keller-Segel solution will blow up in expected finite time.

\subsubsection{The Stationary Solution of the Kinetic system}
\begin{figure}[H]
\includegraphics[scale=0.4]{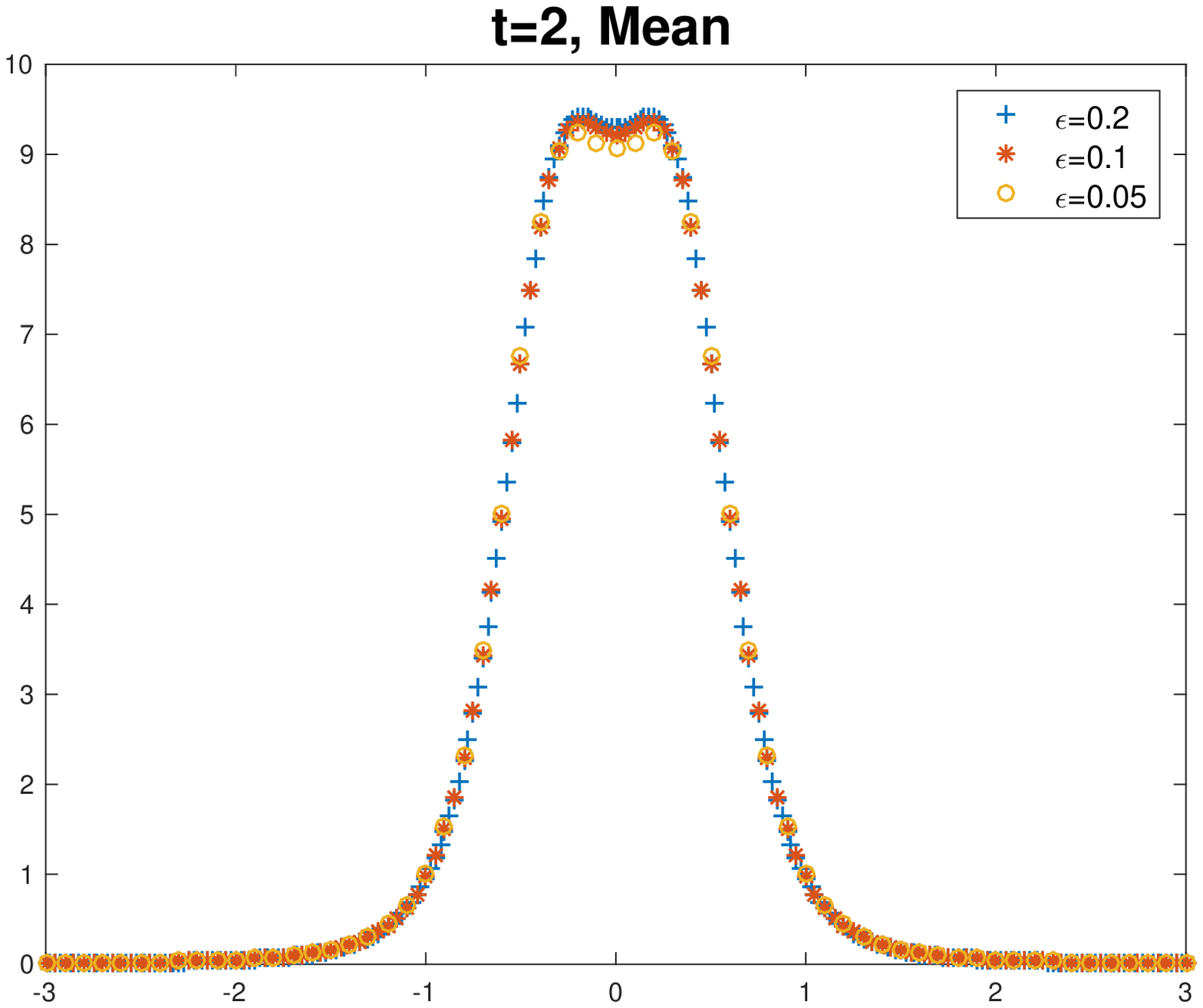}
\includegraphics[scale=0.4]{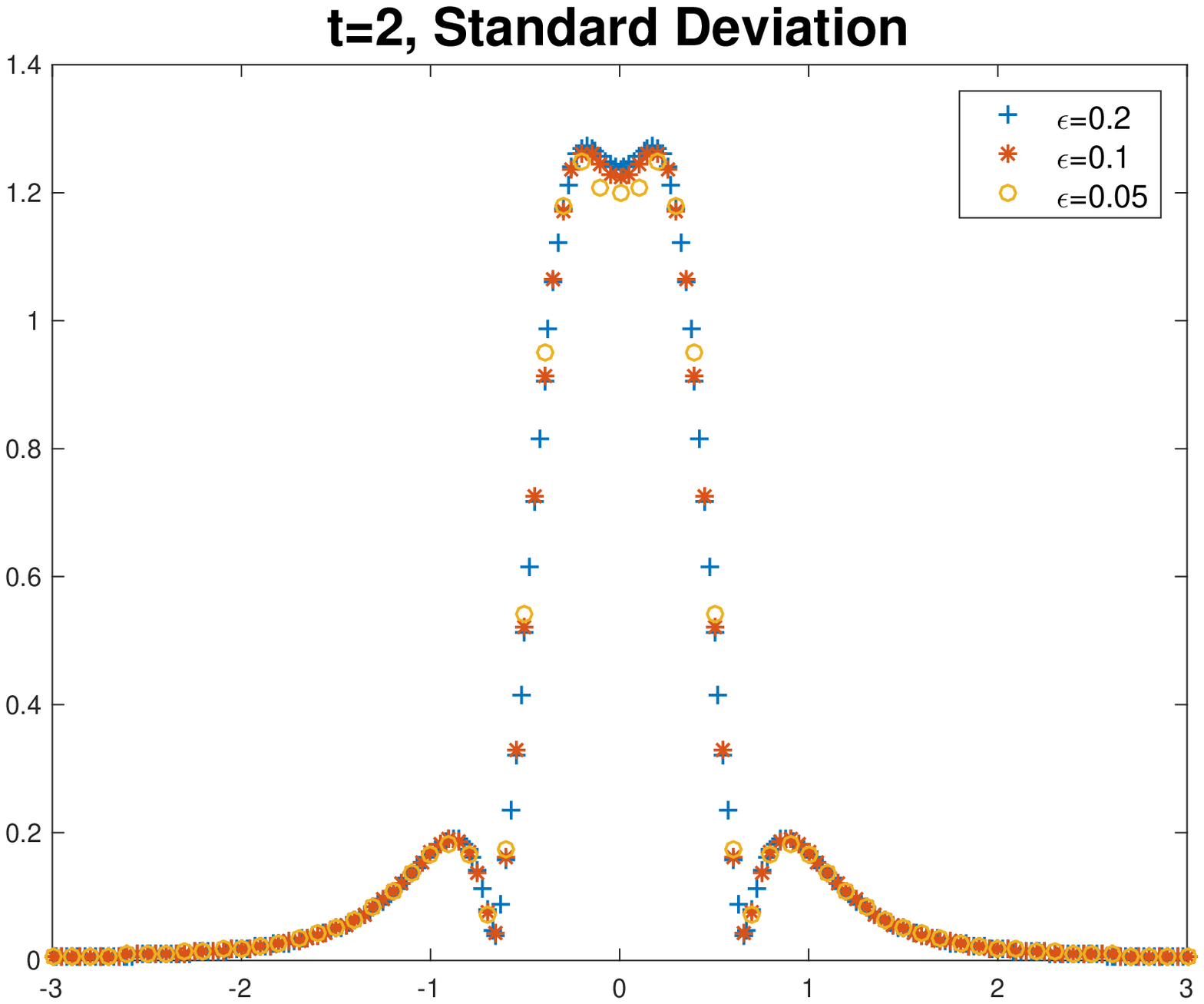}
\caption{The 1D nonlocal random model in the super-critical case. The mean (left) and standard deviation (right) of the function $\varepsilon\rho(\varepsilon x)$ for different $\varepsilon$ are presented. $t=2\gg t_b$.}
\end{figure}

The numerical tests in \cite{carrillo2013asymptotic} suggest that the solution of the deterministic kinetic system with a supercritical initial mass stabilizes toward a stationary state after long time. We also check to see if the same property holds for the kinetic system with random inputs. We plot the mean and standard deviation of $\tilde \rho(x)=\varepsilon\rho(\varepsilon x)$ in Figure 6, which shows that the mean and standard deviation both converge to some stationary state at a long time $t=2$, while the mean agrees with the deterministic stationary solution.

\subsection{The interaction between peaks: the 1D Nonlocal Model with Random Initial Data}
As shown in \cite{BCC}, the interactions between several peaks for the modified Keller-Segel system can be interpreted as optimal transportation. In the following numerical tests, we are going to make some observations of the interaction changes in the kinetic system cased by different types of randomness in initial data. 
\subsubsection{Case 1: Two symmetric peaks, without enough mass in each peak}
\begin{figure}[H]
\begin{center}
\includegraphics[scale=0.4]{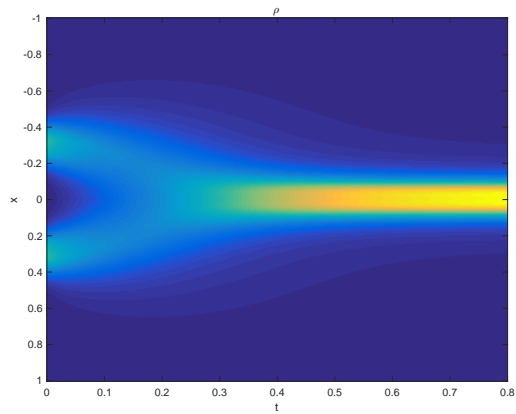}
\caption{ Deterministic solution of $\rho(x,t)$ with initial data $f_0=4\sqrt{5\pi}\left(1.5e^{-80(x-0.3)^2}+1.5e^{-80(x+0.3)^2}\right)$, $\varepsilon=0.1$. (Figure 8 in \cite{carrillo2013asymptotic}).}
\end{center}
\end{figure}
\begin{figure}[H]
\includegraphics[scale=0.4]{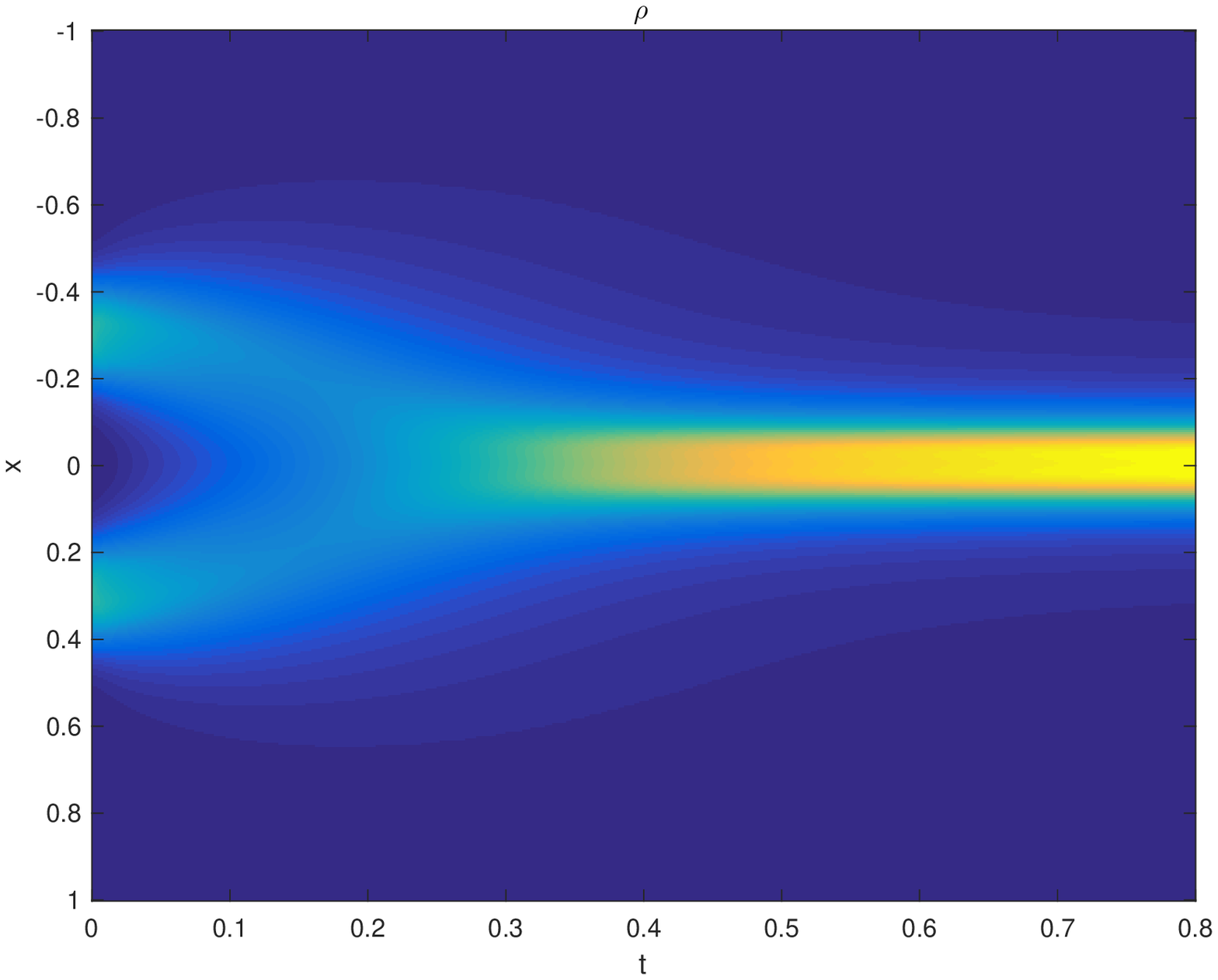}
\includegraphics[scale=0.4]{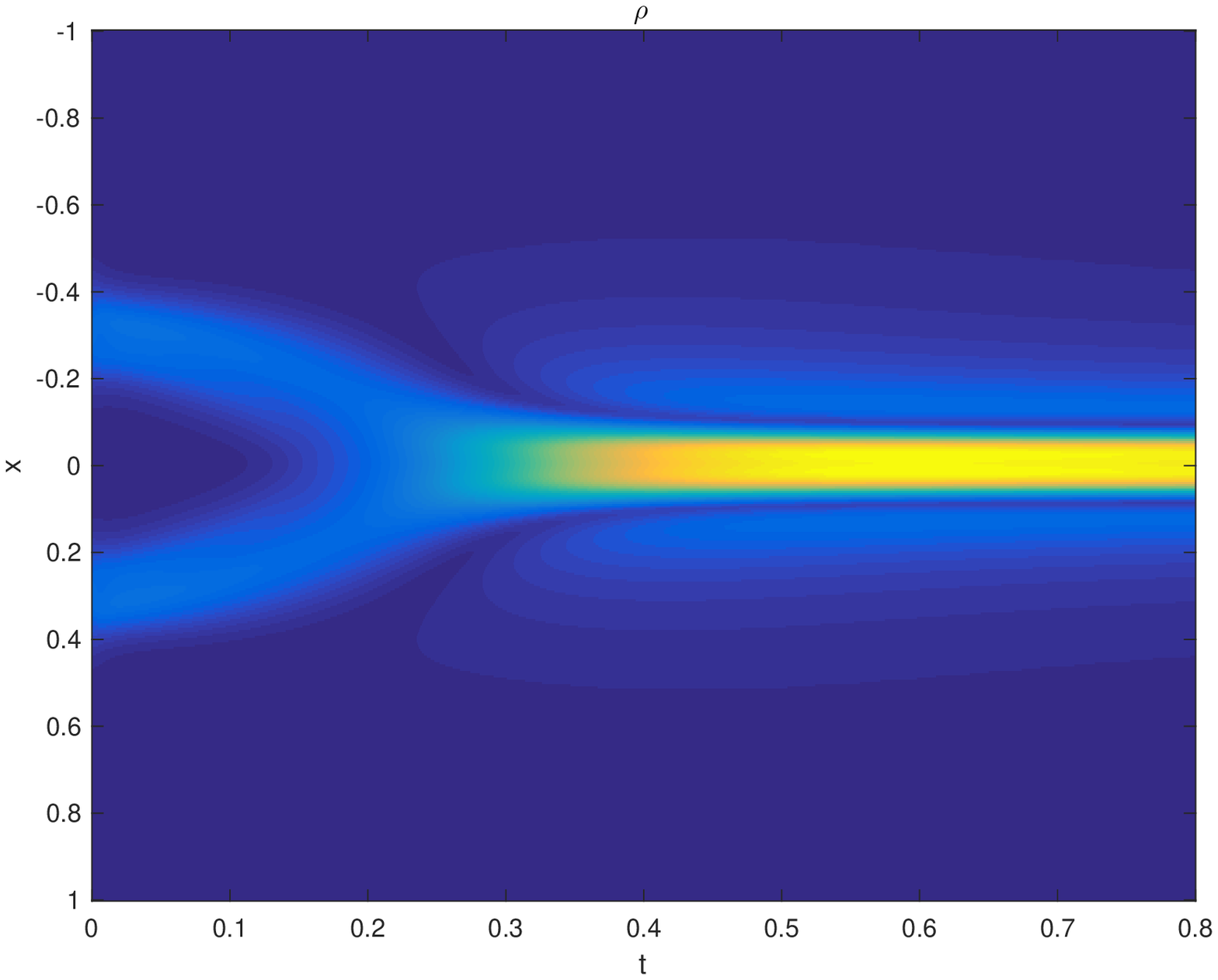}
\caption{Left is the mean and right is the standard deviation of $\rho(x,t,z)$ respectively, with random initial condition $f_0=4\sqrt{5\pi}\left((1.5+0.5z)e^{-80(x-0.3)^2}+(1.5+0.5z)e^{-80(x+0.3)^2}\right), z\sim \mathcal U[-1,1]$, $\varepsilon=0.1$. }
\end{figure}
\begin{figure}[H]
\includegraphics[scale=0.4]{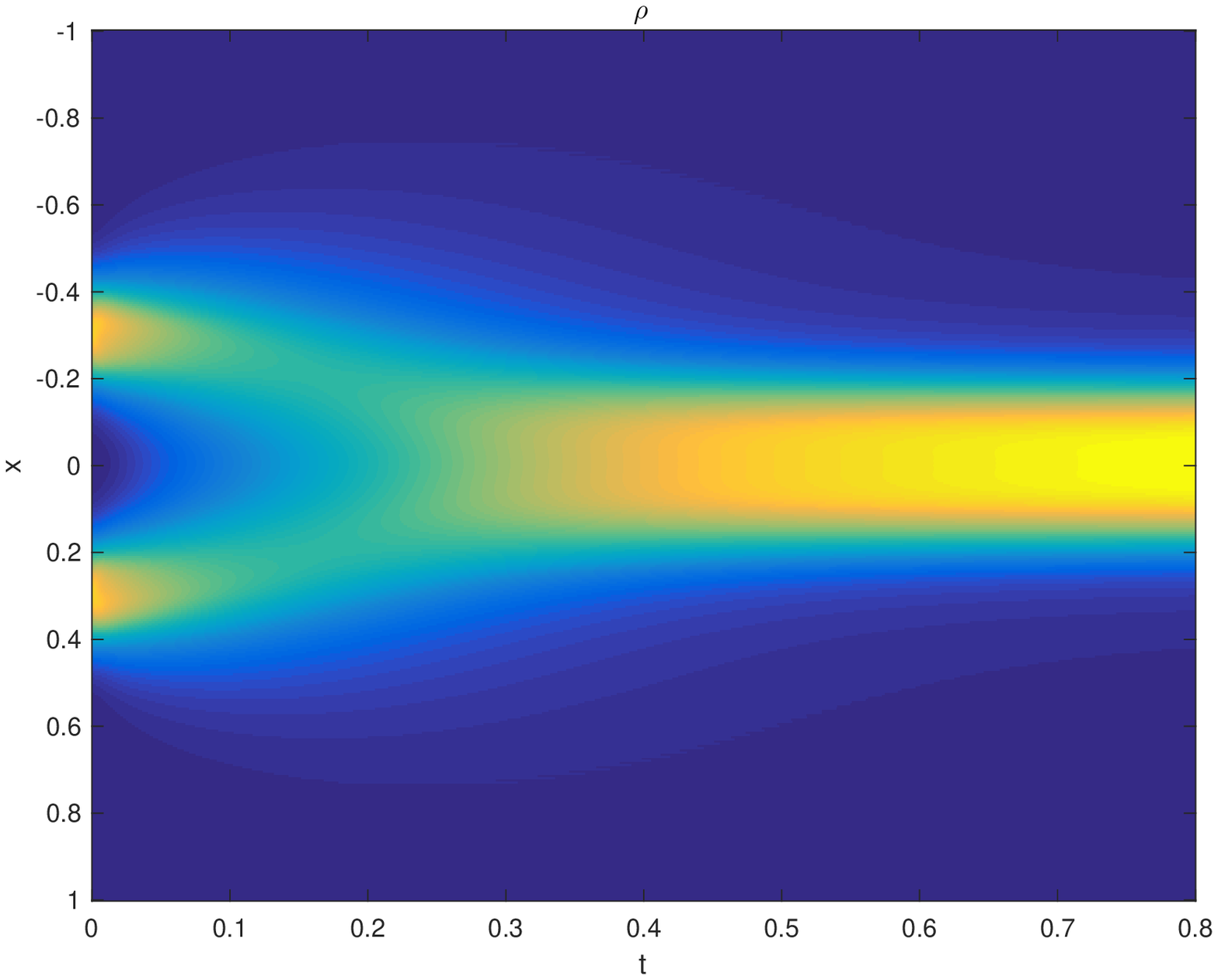}
\includegraphics[scale=0.4]{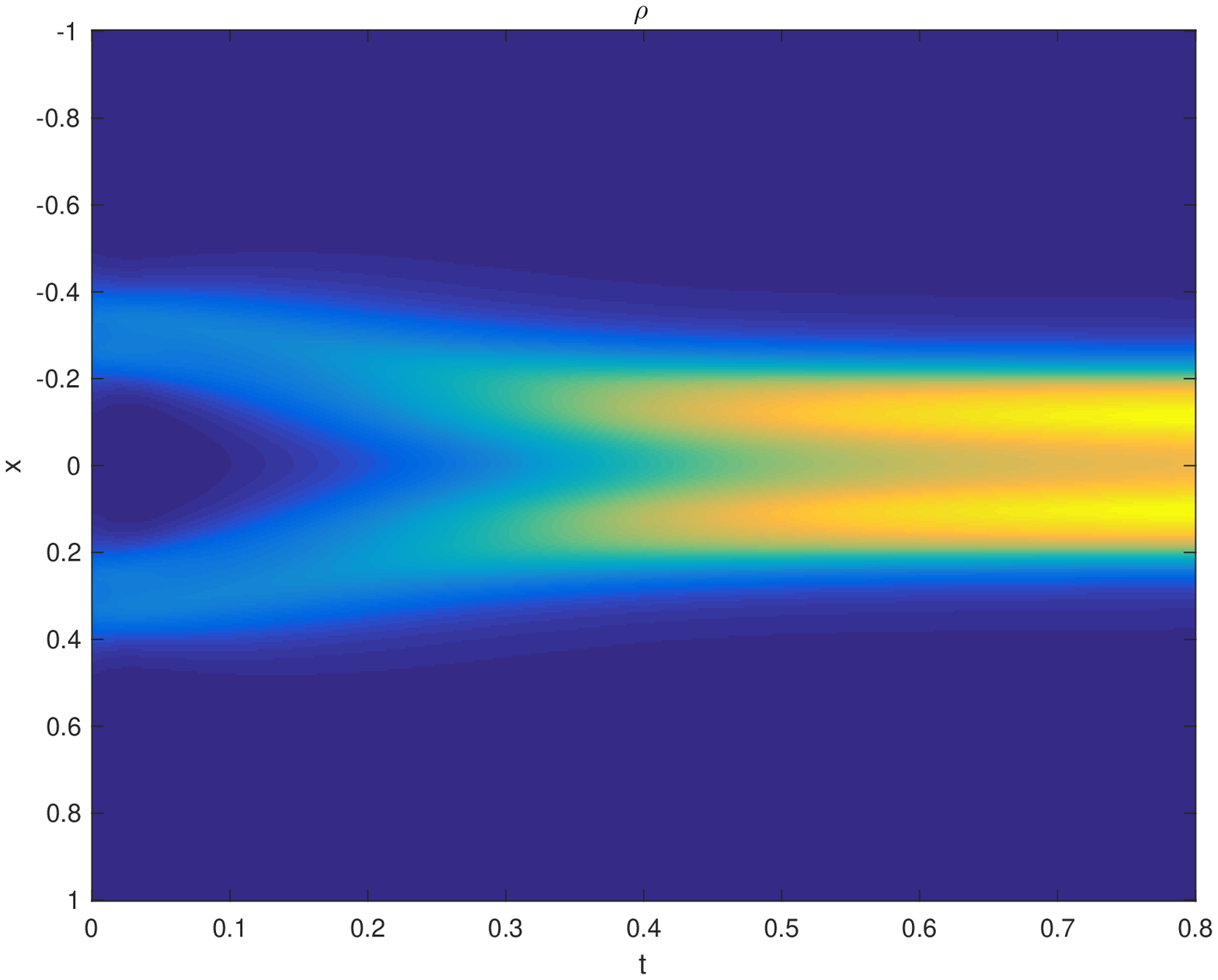}
\caption{Left is the mean and right is the standard deviation of $\rho(x,t,z)$ respectively, with random initial condition $f_0=4\sqrt{5\pi}\left((1.5+0.5z)e^{-80(x-0.3)^2}+(1.5-0.5z)e^{-80(x+0.3)^2}\right), z\sim \mathcal U[-1,1]$, $\varepsilon=0.1$.}
\end{figure}

In this case, we still have $M_c=2\pi$ and $\bar M_c\approx 2.197\pi$. We reproduced the deterministic attraction between two symmetric peaks with total mass $3\pi$ in Figure 7. Then we input symmetric randomness in each peak, i.e. the total mass follows from uniform distribution from $2\pi$ to $4\pi$, keeping each peak without enough mass. Figure 8 shows that symmetric randomness keeps the attraction behavior exactly as the deterministic case. Symmetric properties are preserved both in mean and standard deviation. However, in Figure 9, we input asymmetric randomness in each peak but keeping total mass fixed as $3\pi$. The two peaks will still be attracted in the center but present different behavior as the deterministic one. The asymmetric randomness in this type will widen the mean range of the center peak after concentration, in the sense that asymmetric initial data push the concentrated peak toward the direction with more initial mass.

\subsubsection{Case 2: Two asymmetric peaks with enough mass in each peak}
\begin{figure}[H]
\begin{center}
\includegraphics[scale=0.4]{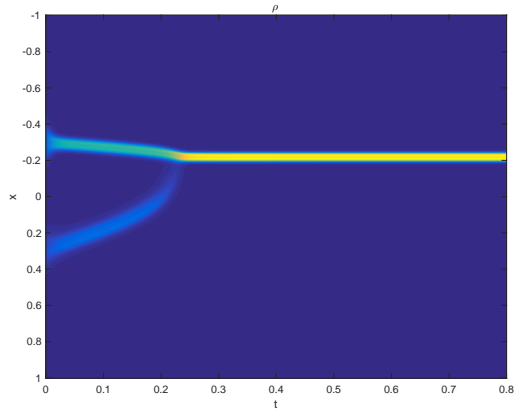}
\end{center}
\caption{Deterministic solution of $\rho(x,t)$ with initial data $f_0=4\sqrt{5\pi}\left(3e^{-80(x-0.3)^2}+5e^{-80(x+0.3)^2}\right)$, $\varepsilon=0.05$. }
\end{figure}
\begin{figure}[H]
\includegraphics[scale=0.4]{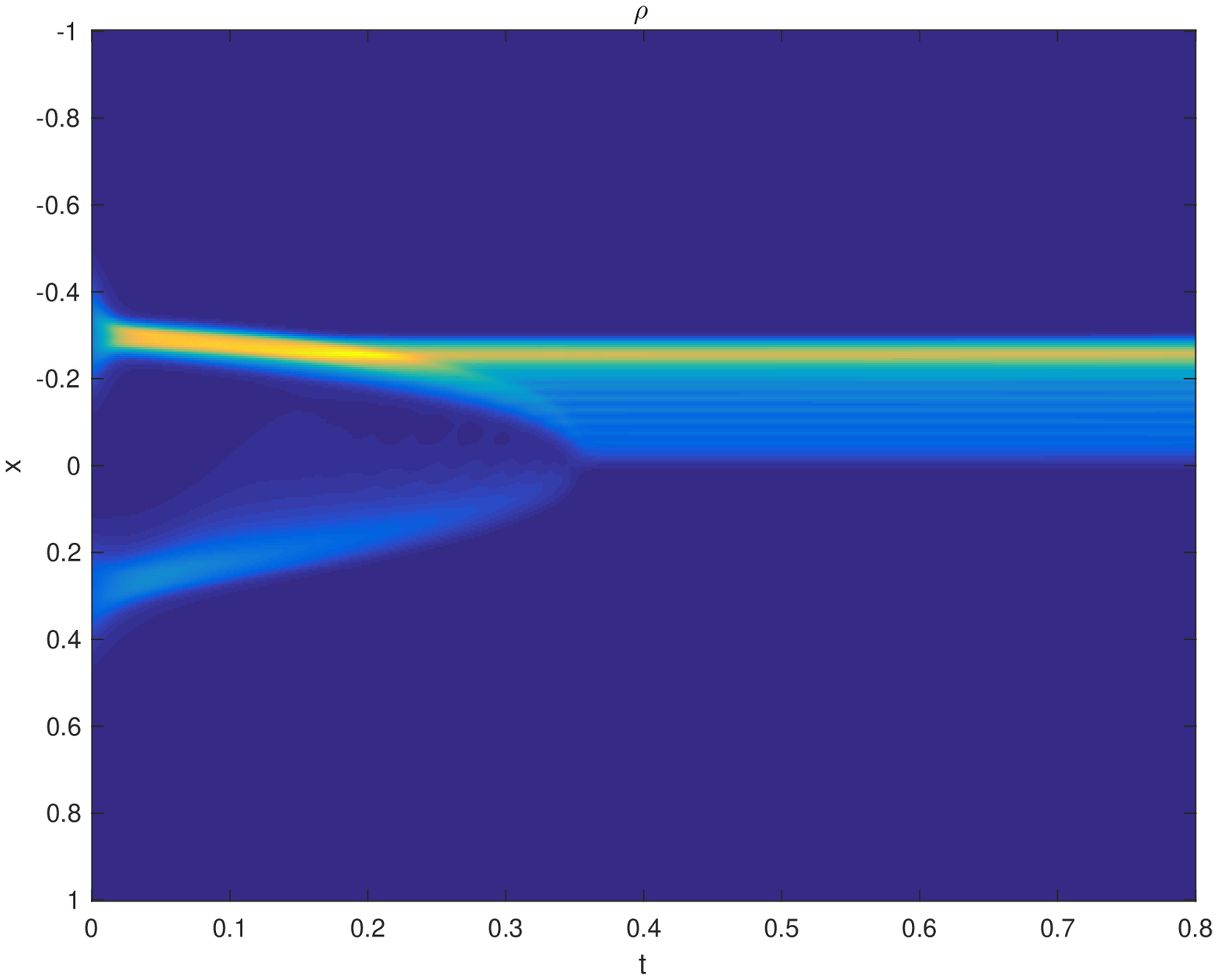}
\includegraphics[scale=0.4]{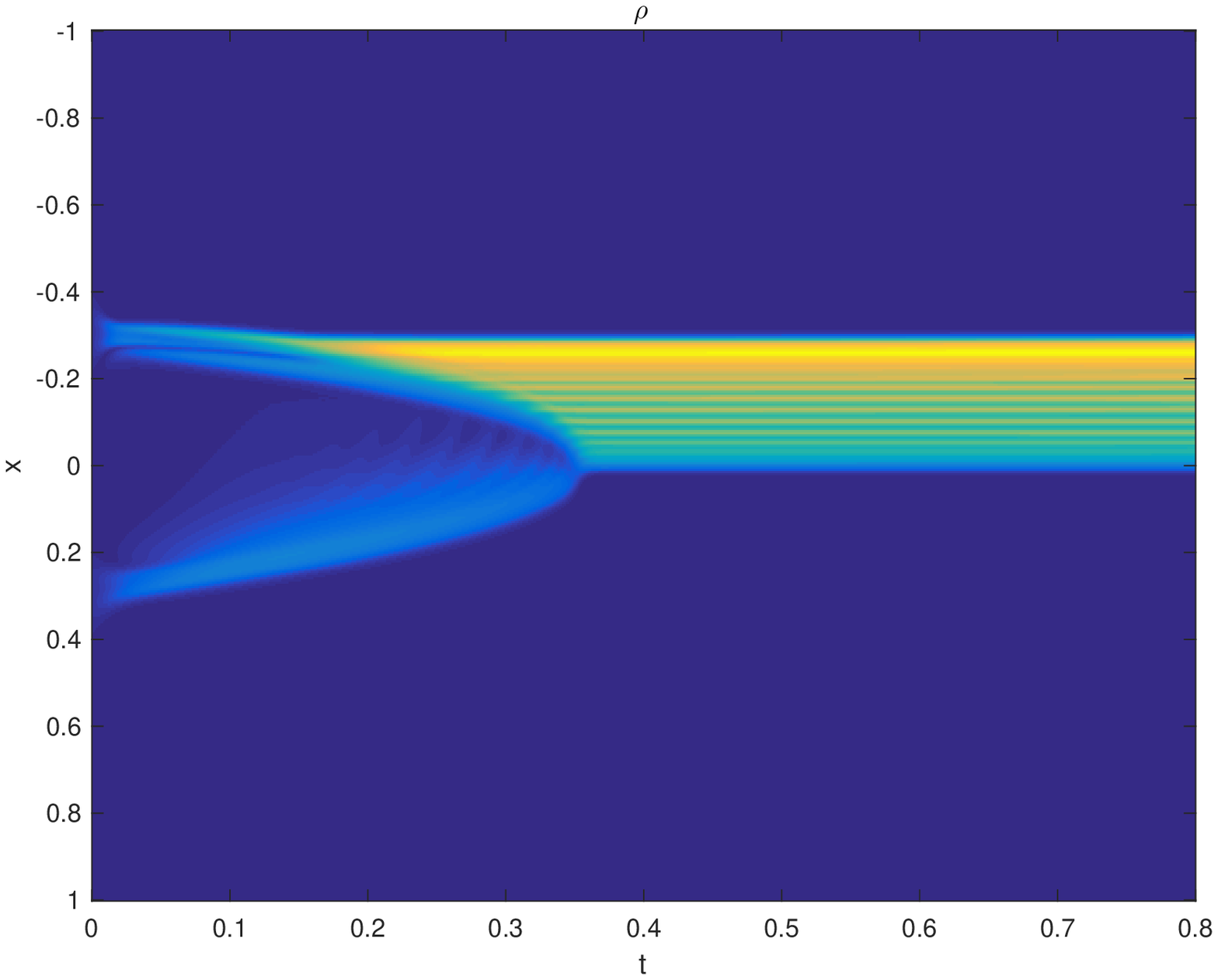}
\caption{Left is the mean and right is the standard deviation of $\rho(x,t,z)$ respectively, with random initial condition $f_0=4\sqrt{5\pi}\left((3+z)e^{-80(x-0.3)^2}+(5-z)e^{-80(x+0.3)^2}\right), z\sim \mathcal U[-1,1]$, $\varepsilon=0.05$. }
\end{figure}

With $M_c=2\pi$ and $\bar M_c\approx 2.197\pi$, we put asymmetric initial mass both larger than $2\pi$. Figure 10 shows similar results as Figure 10 in \cite{carrillo2013asymptotic}. The mass in each peak is large enough to concentrate but they will merge into a larger peak which locates closer to larger initial peak due to asymmetry. Figure 11 shows the effect of the asymmetric randomness with total initial mass fixed. It can be observed in mean and standard deviation that the randomness affects the concentration time, location and asymmetry, showing the solution behaves sensitively to initial data.

\subsubsection{Case 3: Two Asymmetric peaks (close), one below critical mass, one above critical mass}
\begin{figure}[H]
\begin{center}
\includegraphics[scale=0.4]{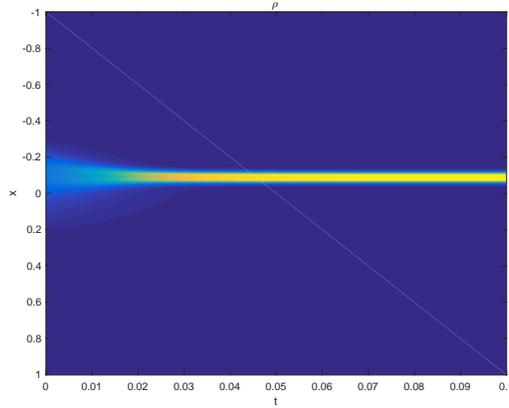}
\end{center}
\caption{Deterministic solution of $\rho(x,t)$ with initial data $f_0=4\sqrt{5\pi}\left(e^{-80(x-0.1)^2}+5e^{-80(x+0.1)^2}\right)$, $\varepsilon=0.05$.}
\end{figure}

\begin{figure}[H]
\includegraphics[scale=0.4]{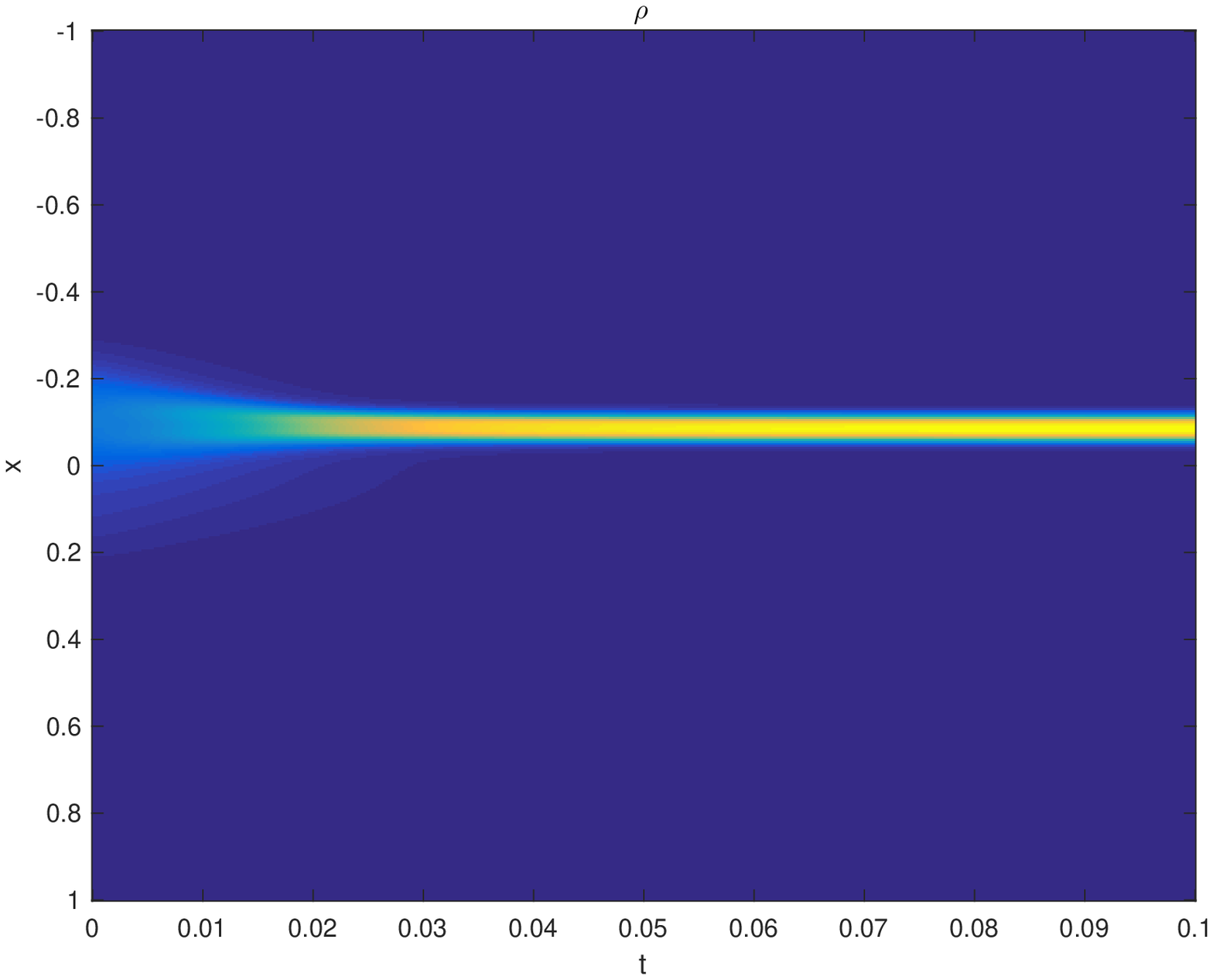}
\includegraphics[scale=0.4]{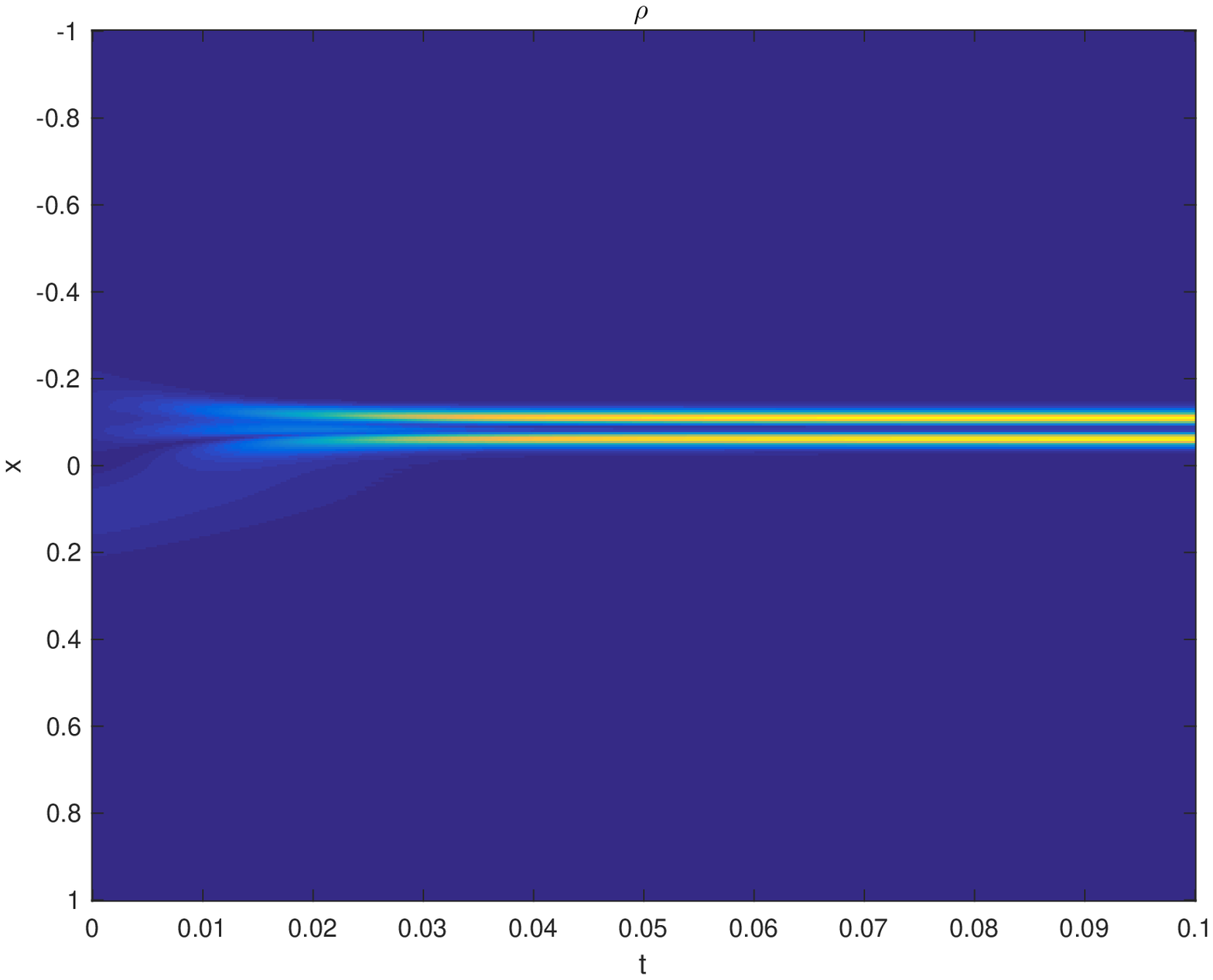}
\caption{Left is the mean and right is the standard deviation of $\rho(x,t,z)$ respectively, with random initial condition $f_0=4\sqrt{5\pi}\left((1+0.5z)e^{-80(x-0.1)^2}+(5-0.5z)e^{-80(x+0.1)^2}\right), z\sim \mathcal U[-1,1]$, $\varepsilon=0.05$.}
\end{figure}

\begin{figure}[H]
\includegraphics[scale=0.4]{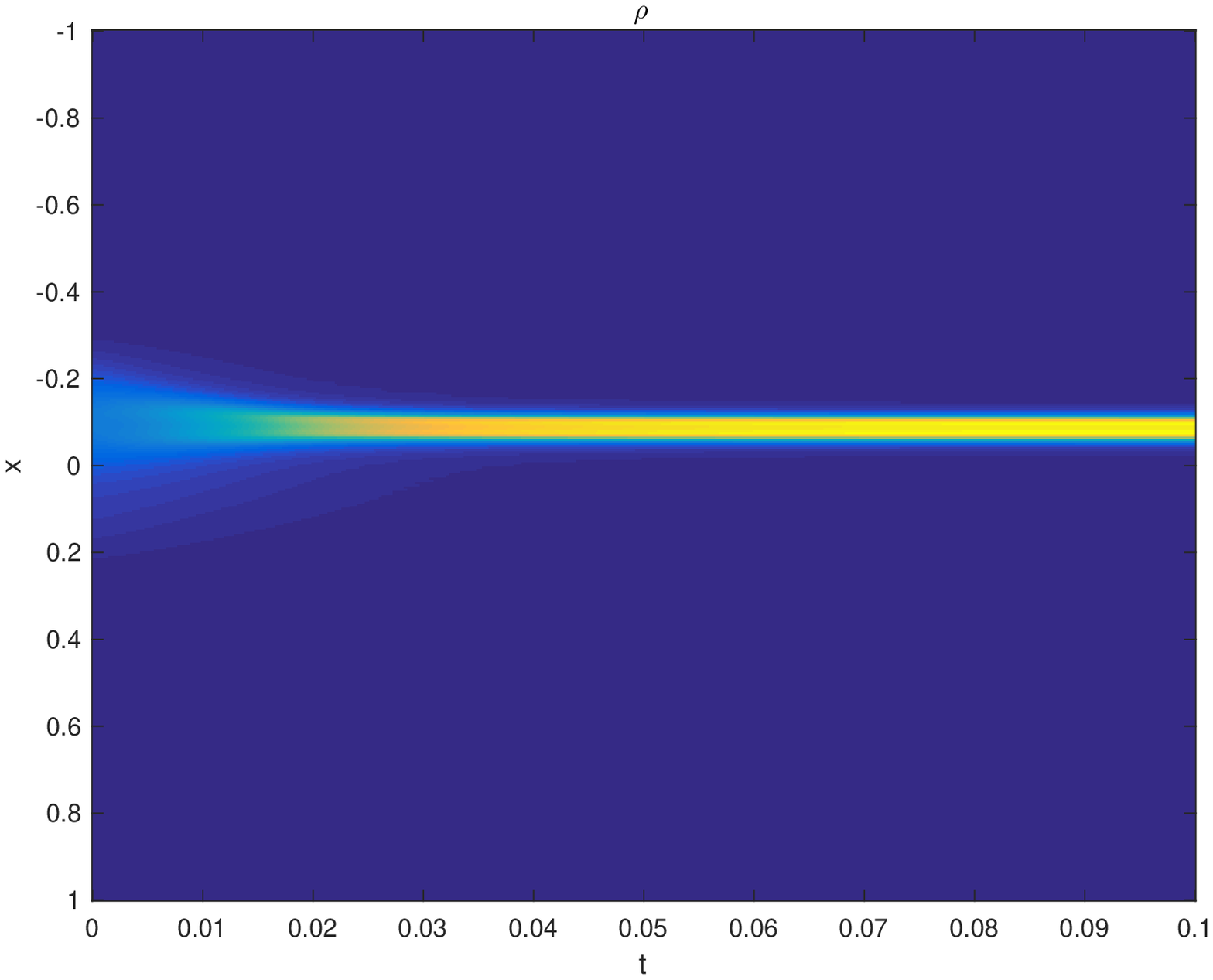}
\includegraphics[scale=0.4]{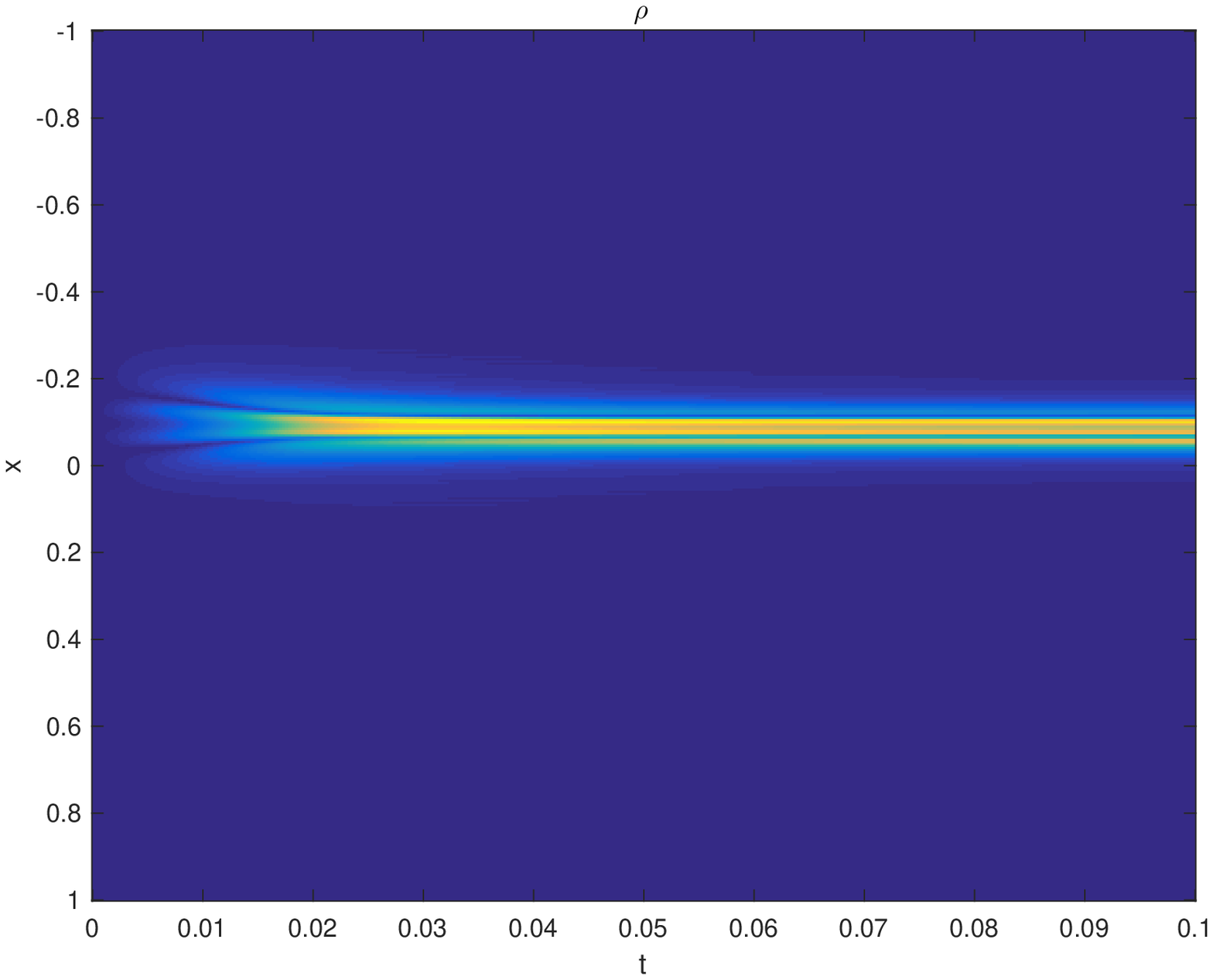}
\caption{Left is the mean and right is the standard deviation of $\rho(x,t,z)$ respectively, with random $\alpha=1+0.5z, z\sim \mathcal U[-1,1]$ and deterministic initial data, $\varepsilon=0.05$.}
\end{figure}

\begin{figure}[H]
\includegraphics[scale=0.4]{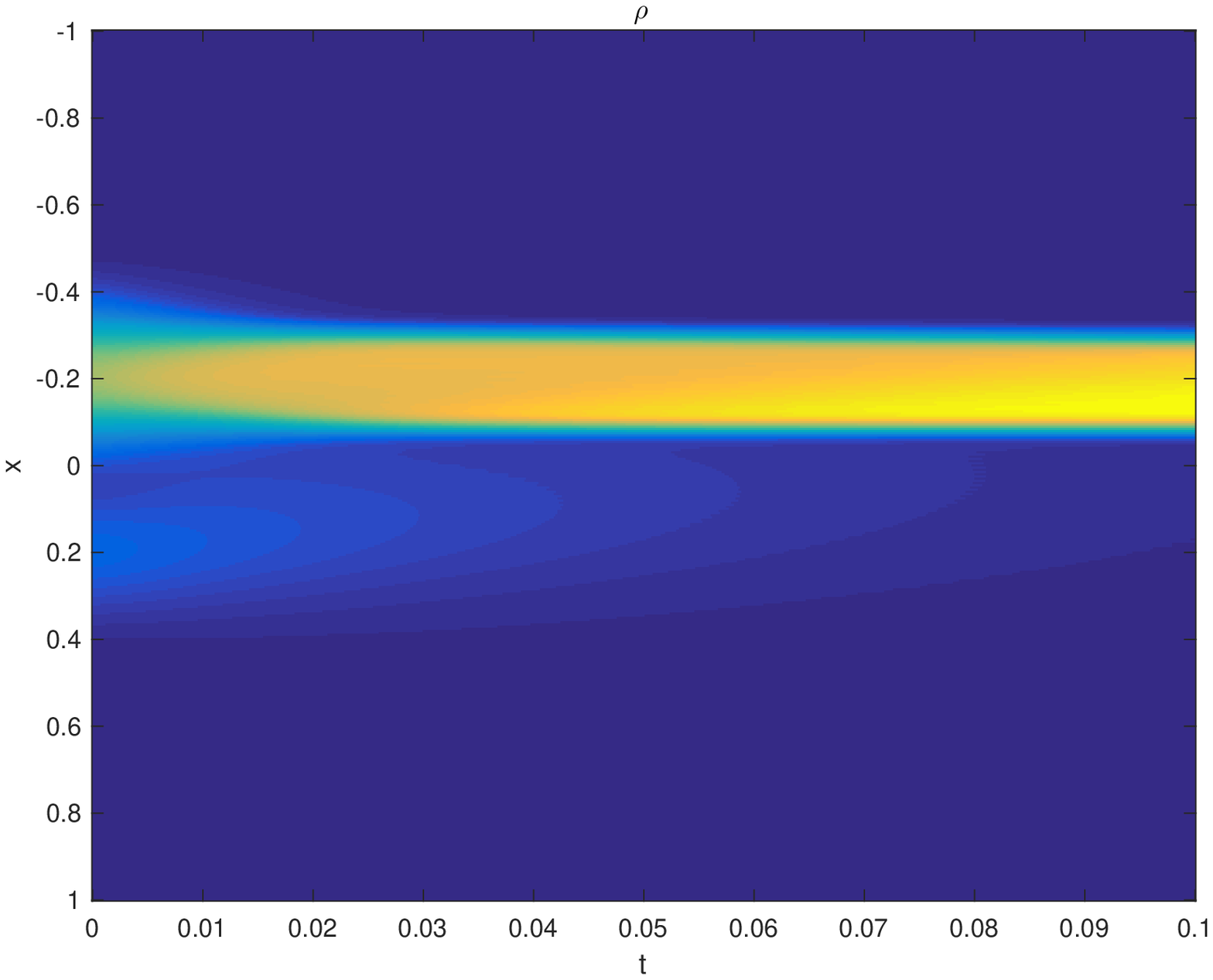}
\includegraphics[scale=0.4]{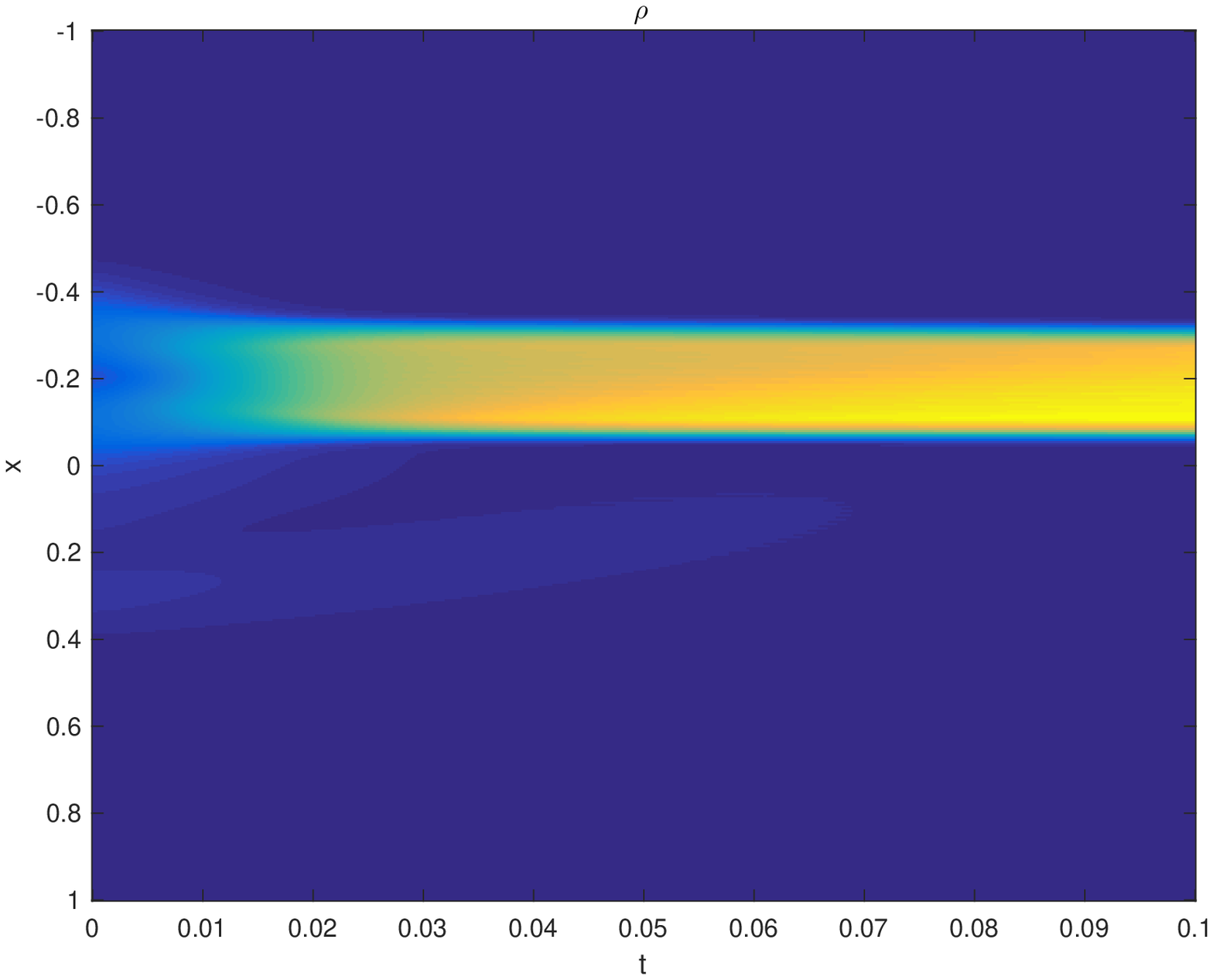}
\caption{Left is the mean and right is the standard deviation of $\rho(x,t,z)$ respectively, with random in position $f_0=4\sqrt{5\pi}\left(e^{-80(x-(0.1+0.1z))^2}+5e^{-80(x+(0.1+0.1z))^2}\right)$, $\varepsilon=0.05$.}
\end{figure}

From Figure 12 to Figure 15, we conduct a series of experiments with two asymmetric peaks, keeping one peak with enough mass and the other one without enough mass. The deterministic case (Figure 12) shows that the peak with less mass will move towards the other one in a short time and then they continue to aggregate mass. In Figure 13, small amount of randomness exchanging between two peaks will not change this tendence in mean. The standard deviation in Figure 13 is asymmetric due to the asymmetric randomness in initial data. In Figure 14, although mean values show no difference, the standard deviation is symmetric because the source of randomness comes from the diffusion coefficient $\alpha$.  Figure 15 shows that the position of the two peaks has significant effects on the aggregation behavior in this case. From mean and standard deviation, one can observe that there exists a critical distance between the two peaks, beyond which the two peaks will not be able to merge. They will be separated to behave independently according to their initial mass.

\subsection{The 1D local Model with Random Initial Data}
\begin{figure}[H]
\begin{center}
\includegraphics[scale=0.4]{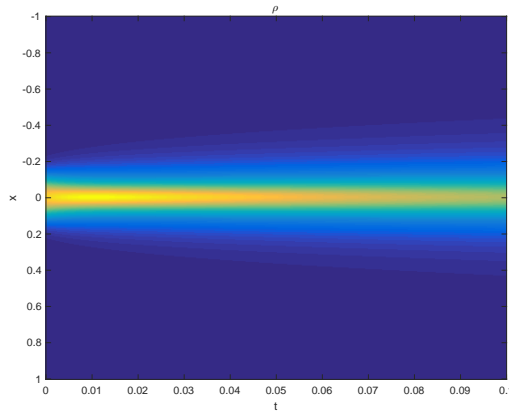}
\end{center}
\caption{Deterministic solution of $\rho(x,t)$ with initial data $f_0=1.5\times4\sqrt{5\pi}e^{-80x^2}$ (subcritical), $\varepsilon=0.01$.}
\end{figure}

\begin{figure}[H]
\includegraphics[scale=0.4]{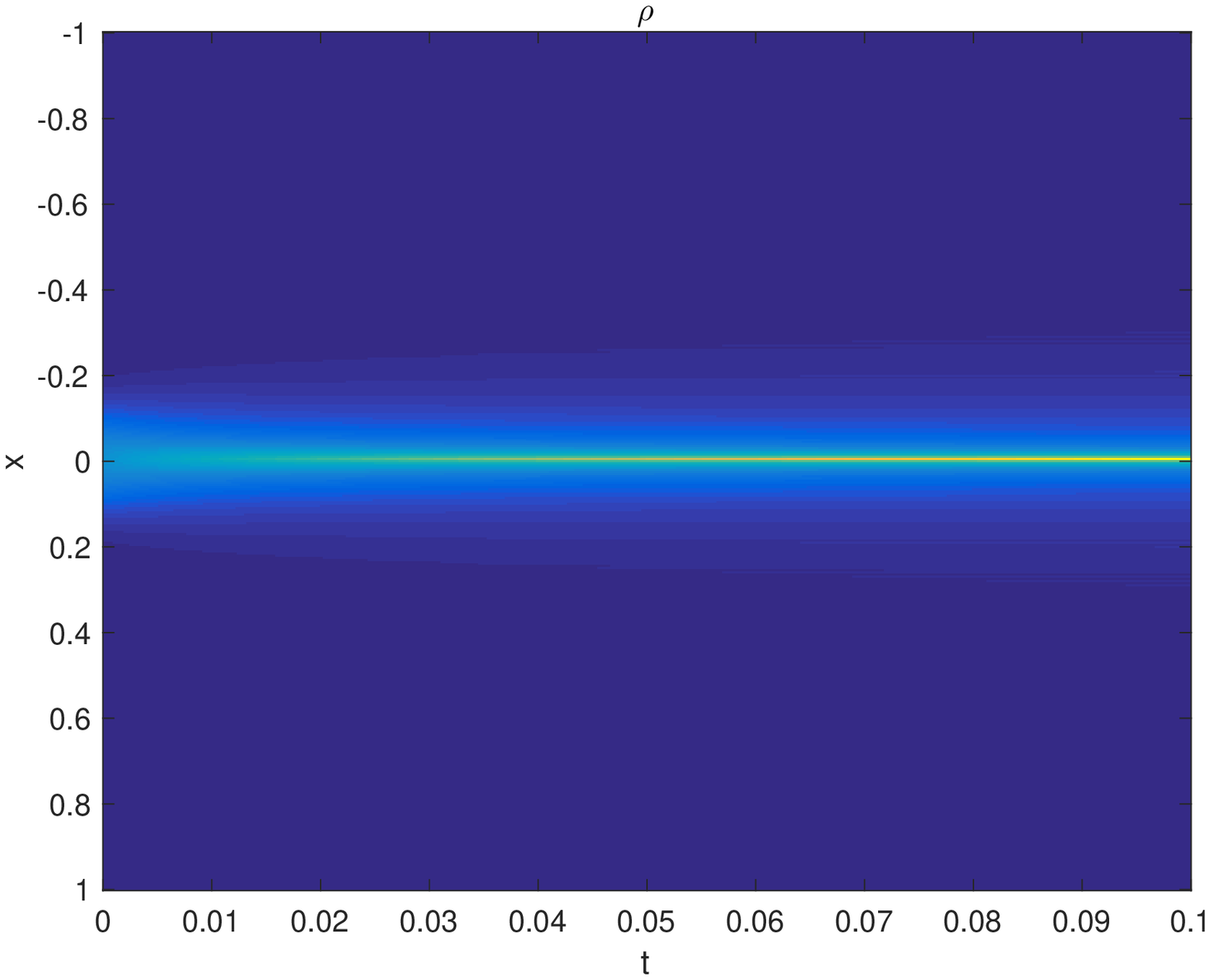}
\includegraphics[scale=0.4]{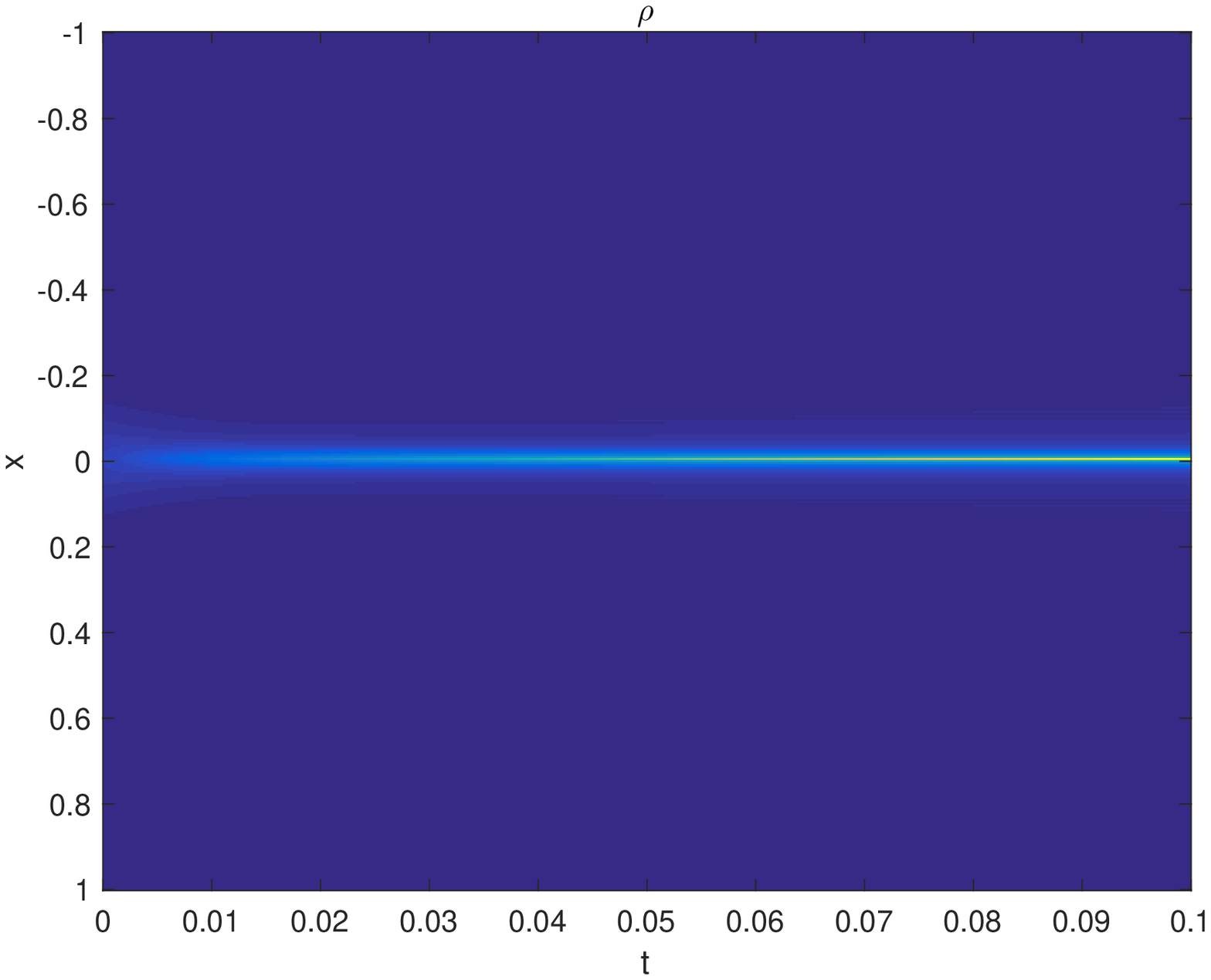}
\caption{Left is the mean and right is the standard deviation of $\rho(x,t,z)$ respectively, with random in initial $f_0=(1.5+z)\times4\sqrt{5\pi}e^{-80x^2}$, $\varepsilon=0.01$.}
\end{figure}

\begin{figure}[H]
\begin{center}
\includegraphics[scale=0.4]{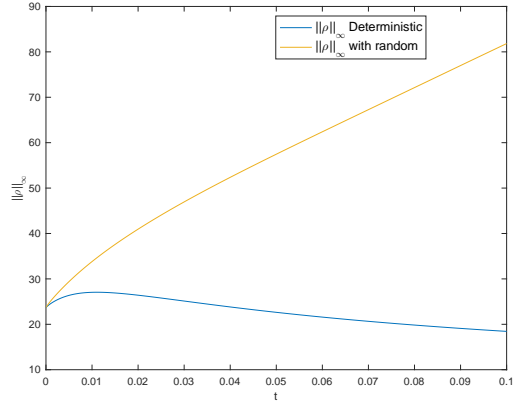}
\end{center}
\caption{Comparison of $\|\rho\|_\infty$ in deterministic solution and mean solution.}
\end{figure}

Although theoretic study of the local model with supercritical mass is still not enough to understand the blow up behavior of the local kinetic chemotaxis system, numerical tests in \cite{carrillo2013asymptotic} suggested blowing up density by using adapted grids. Instead of studying the blowing up property, we are more interested in the sensitive effect brought up by the randomness around critical mass. In Figure 16, the deterministic solution with subcritical initial data will stay bounded as expected from theory. However, the solution keeps aggregating in Figure 17 if we introduce randomness into initial mass ranging from subcritical mass to supercritical mass with mean less than critical mass. More obviously in Figure 18, the deterministic solution will remain bounded while the mean of the random solution appears increasing in time. This indicates that the introduced randomness will influence the properties of the solution. If the range of the initial data contains supercritical regimes, the solution of the random system will behave quite differently from the deterministic one with average initial mass.

\begin{remark}
Stochastic collocation method is used in test 6.4 to deal with $|\partial_xs|$ as following:
Once $\hat{\bold s}=(\hat s_1,\cdots \hat s_K)^T$ is obtained at each time iteration, $\partial_x\hat{\bold s}=(\partial_x\hat s_1,\cdots \partial_x\hat s_K)^T$ can be obtained using finite difference. Then $|\partial_xs(x,z)|$ can be approximated by $|\sum_{k=1}^K\partial_x\hat s_k(x)\Phi_k(z)|$. According to the probability density function of $z$, one can have a set of collocation points $\{z_j\}_{j=1}^M$ with corresponding weights $\{w_j\}_{j=1}^M$. ($M=20$ points are used in our test.) Project $|\partial_xs(x,z)|$ onto the space $\{\Phi_1(z),\cdots,\Phi_K(z)\}$ in order to get the gPC coefficients $(\xi_1,\cdots,\xi_K)^T$ of $|\partial_x s|$ such that $|\partial_xs|\approx\sum_{k=1}^K\xi_k(x)\Phi_k(z)$, one can get 
$$\begin{aligned}
\xi_k(x)&=\int_{I_z}|\partial_xs(x,z)|\Phi_k(z)\lambda(z)dz\\
&\approx\sum_{j=1}^M|\partial_xs(x,z_j)|\Phi_k(z_j)w_j\\
&\approx\sum_{j=1}^M|\sum_{i=1}^K\partial_x\hat s_i(x)\Phi_i(z_j)|\Phi_k(z_j)w_j, \ \ \ \ \ \ k=1,\cdots, K.
\end{aligned}$$
Then $(\xi_1,\cdots,\xi_K)^T$ is used in the algorithm.
\end{remark}

\section{Conclusion}
In this article, a high order efficient stochastic Asymptotic-Preserving scheme is designed for the kinetic chemotaxis system with random inputs. Compared with the previous work \cite{carrillo2013asymptotic} for the deterministic kinetic chemotaxis equations, our new method, based on generalized Polynomial Chaos Galerkin approach to deal with uncertainty, uses the implicit-explicit Runge-Kutta (IMEX-RK) method to gain high accuracy and utilize a macroscopic penalty to improve the CFL stability condition from parabolic type to hyperbolic type in the diffusive regime. Both efficiency and accuracy are verified in the numerical tests.

There are many remaining work for future study. Since the kinetic description of the chemotaxis system is more microscopic and consistent with the classical Keller-Segel equation with more favorable properties (e.g. global existence for nonlocal turning kernel), it is important to complete the theory as well as conduct efficient numerical simulations comparing with experimental results. On one hand, many properties, which have been explored numerically in this paper and previous work \cite{carrillo2013asymptotic,chertockasymptotic}, remain to be verified by rigorous theory. On the other hand, the high order efficient method in this paper should be extended to 2D and 3D kinetic chemotaxis system to support the theory in future work. Moreover, some general problems for uncertainty quantification, such as high dimensionality and rigorous sensitive analysis, are to be further studied.

\renewcommand\refname{Reference}
\bibliographystyle{plain}
\bibliography{ChemotaxisUQr}

\end{document}